\documentclass{amsart}

\usepackage{amssymb}
\usepackage[all]{xy}

\numberwithin{equation}{section}

\def\Z{{\mathbb Z}}
\def\Q{{\mathbb Q}}
\def\C{{\mathbb C}}
\def\P{{\mathbb P}}
\def\E{{\mathbb E}}

\def\H{{\mathbb H}}
\def\L{{\mathbb L}}
\def\V{{\mathbb V}}

\def\A{{\mathcal A}}

\def\cC{{\mathcal C}}
\def\D{{\mathcal D}}
\def\E{{\mathcal E}}

\def\cG{{\mathcal G}}
\def\cH{{\mathcal H}}

\def\M{{\mathcal M}}
\def\cN{{\mathcal N}}

\def\cP{{\mathcal P}}

\def\T{{\mathcal T}}
\def\U{{\mathcal U}}
\def\cV{{\mathcal V}}

\def\e{\epsilon}

\def\G{\Gamma}

\def\d{{\mathfrak d}}
\def\e{{\mathfrak e}}

\def\g{{\mathfrak g}}
\def\h{{\mathfrak h}}

\def\n{{\mathfrak n}}
\def\p{{\mathfrak p}}
\def\r{{\mathfrak r}}
\def\s{{\mathfrak s}}

\def\u{{\mathfrak u}}

\def\etabar{{\overline{\eta}}}

\def\Ql{{\Q_\ell}}
\def\Zl{{\Z_\ell}}

\def\Qp{{\Q_p}}

\def\Gm{{\mathbb{G}_m}}
\def\Sp{{\mathrm{Sp}}}

\def\GL{{\mathrm{GL}}}
\def\GSp{{\mathrm{GSp}}}

\def\arith{\mathrm{arith}}

\def\geom{\mathrm{geom}}

\def\top{\mathrm{top}}
\def\orb{\mathrm{orb}}

\def\ab{\mathrm{ab}}

\def\et{\mathrm{\acute{e}t}}

\def\nab{\mathrm{nab}}

\def\hyp{\mathrm{hyp}}

\def\tambo{\boxplus}

\newcommand\im{\operatorname{im}}
\newcommand\id{\operatorname{id}}
\newcommand\ad{\operatorname{ad}}

\newcommand\Hom{\operatorname{Hom}}

\newcommand\Spec{\operatorname{Spec}}

\newcommand\Aut{\operatorname{Aut}}
\newcommand\Inn{\operatorname{Inn}}
\newcommand\Out{\operatorname{Out}}
\newcommand\Der{\operatorname{Der}}

\newcommand\Gr{\operatorname{Gr}}

\newcommand\Jac{\operatorname{Jac}}

\newcommand\Gal{\operatorname{Gal}}
\newcommand\Pic{\operatorname{Pic}}

\newcommand\Sect{\operatorname{Sect}}

\newtheorem{theorem}{Theorem}[section]
\newtheorem{lemma}[theorem]{Lemma}
\newtheorem{proposition}[theorem]{Proposition}
\newtheorem{corollary}[theorem]{Corollary}
\newtheorem{bigtheorem}{Theorem}

\theoremstyle{definition}

\theoremstyle{remark}
\newtheorem{remark}[theorem]{Remark}

\newtheorem{variant}[theorem]{Variant}

\begin{document}


\title{On the sections of universal hyperelliptic curves}


\author{Tatsunari Watanabe}
\address{Department of Mathematics, Purdue University, 
West Lafayette}
\email{twatana@purdue.edu}
\urladdr{www.math.purdue.edu/people/bio/twatana} 





\begin{abstract} In this paper, we will give an algebraic proof for determining the sections for the universal pointed  hyperelliptic curves $\cC_{\cH_{g,n/k}}\to \cH_{g,n/k}$, when $g\geq 3$ and the image of the $\ell$-adic cyclotomic  character $G_k\to \Z^\times$ is infinite. Furthermore, we will study the nonabelian phenomena associated to the universal hyperelliptic curves. For example, we will show that the section conjecture holds for $\cC_{\cH_{g/k}}\to \cH_{g/k}$ and the unipotent analogue of the conjecture holds for the pointed cases. This work is an extension of Hain's original work \cite{hain2} to the hyperelliptic case. 

\end{abstract}


 \maketitle



\section{Introduction}
Let $k$ be a field of characteristic zero. Denote by $\M_{g,n/k}$ the moduli stack of proper smooth $n$-pointed curves of genus $g$ over $k$. Suppose that $2g-2+n>0$. When $k\subset \C$, it is known that if $g\geq 3$, then the sections of the universal curve over $\M_{g,n/k}$ are exactly the $n$ tautological sections. This follows from results in Teichm$\ddot{\text{u}}$ller theory by Hubbard \cite{Hu} for $n=0$ and Earle and Kra \cite{EaKr} for $n>0$. In \cite[Thm.~1]{hain2}, Hain gave a proof of this fact by using the theory of weighted completion of profinite groups developed by Hain and Matsumoto in \cite{wei}. In this paper, we will consider the universal hyperelliptic curve $\cC_{\cH_{g,n/k}}\to \cH_{g,n/k}$ that is the restriction of the universal curve to the hyperelliptic locus $\cH_{g,n/k}$ in $\M_{g,n/k}$.  It has the involution $J:\cC_{\cH_{g,n/k}}\to \cC_{\cH_{g,n/k}}$ whose restriction to each fiber is the hyperelliptic involution. The tautological sections $x_1,\ldots,x_n$ of $\cC_{\cH_{g,n/k}}\to \cH_{g,n/k}$ have their hyperelliptic conjugates $J\circ x_1$,\ldots, $J\circ x_n$, respectively.  Firstly, we determine the sections of the universal hyperelliptic curve by using the weighted completion of the \'etale fundamental group $\pi_1(\cH_{g,n/k})$ of $\cH_{g,n/k}$. 
\begin{bigtheorem} Let $k$ be a field of characteristic zero such that the image of $\ell$-adic cyclotomic character $\chi_\ell:G_k\to \Zl^\times$ is infinity for a prime number $\ell$. 
	If $g\geq 3$, then the sections of  $\cC_{\cH_{g,n/k}}\to \cH_{g,n/k}$ are exactly the tautological ones and their hyperelliptic conjugates. 
\end{bigtheorem}
Our second result concerns with a nonabelian phenomenon  associated with the universal hyperelliptic curves. Let $C$ be a proper smooth curve over $k$ of arithmetic genus $g$. Let $\bar x$ be a geometric point of $C$. Fix an algebraic closure $\bar k$ of $k$ and set $\overline{C}:=C\otimes_k\bar k$. Then there is an exact sequence of profinite groups
$$1\to \pi_1(\overline{C}, \bar x)\to \pi_1(C, \bar x)\to G_k\to 1,$$
where $G_k$ is the absolute Galois group $\Gal(\bar k/k)$.
Set $\Pi=\pi_1(\overline{C},\bar x)$. Each $k$-rational point $x$ of $C$ induces a section $s_x$ of $\pi_1(C,\bar x)\to G_k$ that is well-defined up to conjugation by an element of $\Pi$ and hence a class $[s_x]$ in the set $\T_{\pi_1(C/k)}$ of the $\Pi$-conjugacy classes of sections of $\pi_1(C, \bar x)\to G_k$. The section conjecture predicts that if $k$ is finitely generated over $\Q$ and $g\geq 2$, the function $x\mapsto [s_x]$ is a bijection. Let $K=k(\M_g)$ be the function field of $\M_{g/k}$.  Hain proved in \cite[Thm.~2]{hain2} that if the image of $\ell$-adic cyclotomic character $\chi_\ell:G_k\to \Z_\ell^\times$ is infinite and if $g\geq 5$, the section conjecture holds for the generic curve of genus $g$. It follows from the results in \cite{hain2} that the same result also holds for the universal curve $\cC_{g/k}\to \M_{g/k}$ for $g\geq 4$. Our second result is the hyperelliptic analogue of this fact. Fix a geometric point $\etabar$ of $\cH_{g/k}$ and let $C_\etabar$ be the fiber of $\cC_{\cH_{g/k}}\to \cH_{g/k}$ over $\etabar$. Fix a geometric point $\bar x$ of $C_\etabar$. 
\begin{bigtheorem} Let $k$ be a field of characteristic zero and let $\ell$ be a prime number. 
If $g\geq 3$ and the $\ell$-adic cyclotomic character $\chi_\ell:G_k\to \Z^\times_\ell$ is infinite, then the sequence
$$1\to \pi_1(C_\etabar, \bar x)\to \pi_1(\cC_{\cH_{g/k}}, \bar x)\to \pi_1(\cH_{g/k}, \etabar)\to 1$$ does not split. 
\end{bigtheorem}

Except the empty case, there is no known curve $C$ for which the conjecture holds. Therefore, it is an interesting question to determine the group $\T_{\pi_1(C/k)}$  when $C$ has a $k$-rational point.  
Hain also showed in \cite[Thm.~3]{hain2} that the unipotent version of the section conjecture holds for the generic curve of type $(g,n)$ when $k$ is a number field or a finite extension of $\Q_p$ and $g\geq 5$. In the unipotent case, we will consider extensions of a profinite group by a prounipotent group $\U$ and the $\U$-conjugacy classes of splittings of them.  The study of such nonabelian cohomology theory was developed by Kim in \cite{Kim}, and furthermore in \cite{hain4}, Hain extends the theory for extensions of a proalgebraic group $\cG$ by a prounipotent group $\U$ and defines the nonabelian cohomology scheme $H^1_\nab(\cG, \U)$, which is an affine scheme under a mild condition. Our third result is the hyperelliptic analogue of Hain's result above and its proof uses the nonabelian cohomology scheme of the weighted completion $\D_{g,n}$ of $ \pi_1(\cH_{g,n/k})$.  Let $\etabar$ be a geometric point of $\cH_{g,n/k}$. Denote by $\cP$ the unipotent completion of $\pi_1(C_\etabar, \bar x)$ over $\Ql$. There is a group extension of $\pi_1(\cH_{g,n/k}, \etabar)$ by $\cP(\Ql)$
$$1\to \cP(\Ql)\to G\to \pi_1(\cH_{g,n/k}, \etabar)\to 1$$
such that the diagram
$$\xymatrix{
	1\ar[r]&\pi_1(C_\etabar, \bar x)\ar[r]\ar[d]&\pi_1(\cC_{\cH_{g,n/k}},\bar x)\ar[r]\ar[d]&\pi_1(\cH_{g,n/k}, \etabar)\ar[r]\ar@{=}[d]&1\\
	1\ar[r]&\cP(\Ql)\ar[r]&G\ar[r]&\pi_1(\cH_{g,n/k}, \etabar)\ar[r]&1
}
$$ commutes, where the left-hand vertical map is the canonical map obtained by unipotent completion. The group $G$ is uniquely determined by the weighted completions of the projection $\pi_1(\cC_{\cH_{g,n/k}})\to \pi_1(\cH_{g,n/k})$. We define $H^1_\nab(\pi_1(\cH_{g,n/k}, \etabar), \cP)(\Ql)$ to be the set of $\cP(\Ql)$-conjugacy classes of sections of $G\to \pi_1(\cH_{g,n/k}, \etabar)$. 
\begin{bigtheorem}
	Let $p$ be a prime number distinct from $\ell$. Let $k$ be a number field or finite extension of $\Qp$. If $g\geq 3$, then there is a bijection between the set of the sections of $\cC_{\cH_{g,n/k}}\to \cH_{g,n/k}$ and $H^1_\nab(\pi_1(\cH_{g,n/k},\etabar), \cP)(\Ql)$. 
\end{bigtheorem}

{\em Acknowledgments:} Some of the computations appearing in this paper had been done in 2013 and included in my thesis in 2015 at Duke University.  I am very grateful to my advisor Richard Hain who suggested this problem to me as a graduate student and gave me endless support both in mathematically and personally throughout my graduate life at Duke University. I am also grateful to Makoto Matsumoto who gave me many advises on the study of the arithmetic mapping class groups. I also would like to thank Kevin Kordek and Dan Petersen for numerous discussions on the hyperelliptic mapping class groups. 

\section{A review of Hain's approach and an outline for the hyperellipc case}
Since our main results are the extensions of Hain's work \cite{hain2} to the universal hyperelliptic curves, we will briefly discuss an outline for our main results along Hain's approach and ideas in this section. The main tool that Hain uses is the weighted completion of a profinite group. It is a variant of the relative completion of a discrete group by Deligne.  Associated to a proper smooth family $f: C\to T/k$ of curves of genus $g\geq 2$, there is a monodromy representation 
$$\rho_{\bar x}:\pi_1(T, \bar x)\to  R:=\GSp(H^1_\et(C_{\bar x}, \Ql(1))),$$
where $C_{\bar x}$ is the fiber of $f$ over the geometric point $\bar x$.  The central cocharacter $\omega:\Gm\to R$ is defined by $ a\mapsto a^{-1}\id$. Assuming that the image of $\rho_{\bar x}$ is Zariski-dense in $R$, we have the weighted completion $\cG_T$ of $\rho_{\bar x}$. It is a proalgebraic group over $\Ql$ that is an extension of $R$ by a prounipotent $\Ql$-group $\U_T$. By Levi's theorem, we may choose a splitting $R\to\cG_T$, and hence each finite-dimensional $\cG$-module will admit a natural weight filtration via $\omega$.  Applying the weighted completion to the projection $\pi_1(C)\to \pi_1(T)$, we obtain an exact sequence of proalgebraic $\Ql$-groups
$$1\to \cP\to \cG_C\to \cG_T\to 1,$$
where $\cP$ is the unipotent completion of $\pi_1(C_{\bar x})$ over $\Ql$. Let $\ast\in\{C,T\}$. Denote the Lie algebras of $\cP$, $\cG_\ast$, $\U_\ast$, and $R$ by $\p$, $\g_\ast$, $\u_\ast$, and $\r$, respectively. The adjoint action of $\cG_T$ on the Lie algebra $\g_\ast$ induces a natural weight filtration $W_\bullet\g_T\ast$ satisfying the properties
\begin{enumerate}
	\item $W_{-1}\g_\ast=W_{-1}\u_\ast, W_0\g_\ast=\g_\ast, \text{ and } \Gr^W_0\g_\ast=\r.$
	\item the adjoint action of $\cG_\ast$ on each $\Gr^W_m\u_\ast$ factors through the action of $R$. 
\end{enumerate}
Furthermore, the adjoint action of $\cG_T$ on $\p$ induces a natural weight filtration $W_\bullet\p$, which coincides with the lower central series of $\p$. 
 Hain applies the weighted completion to  the universal family $f:\cC_{g,n}\to \M_{g,n/k}$ for $g\geq 3$, where $k$ is a field of characteristic zero such that the $\ell$-adic cyclotomic character $\chi_\ell:G_k\to \Zl^\times$ is infinite. \\
\indent The first essential part of Hain's work lies in the determination of the $R$-invariant sections of the two-step graded nilpotent Lie algebras
$$\Gr^W_\bullet df_\ast/W_{-3}: \Gr^W_\bullet\u_{C}/W_{-3}\to\Gr^W_\bullet\u_T/W_{-3}.$$
Johnson's fundamental work \cite{joh1} implies that when $n=1$, there is an $R$-invariant isomorphism
$$\Gr^W_{-1}\u_C\cong H_1(\u_C)\cong \Lambda^3_0H\oplus H,$$
where $H=H^1_\et(C_{\bar x}, \Ql(1))$ and $\Lambda^3_0H$ is the quotient $(\Lambda^3H)(-1)$ by the submodule $H\wedge \theta$ with $\theta$ the polarization in $(\Lambda^2H^\ast)(1)$. A key point for determining the sections is that the bracket $[~,~]:\Lambda^2 \Lambda^3_0H\to \Gr^W_{-2}\p$ is nontrivial, that is, the bracket of the outer component maps nontrivially into the inner part of the Lie algebra.  In the hyperelliptic case, we will need to consider down to weight $-4$ and produce the analogue of the bracket computation $[~,~]:\Lambda^2\tambo \to \Gr^W_{-4}\p$, where the representation $\tambo$ is the highest weight part of $(\mathrm{Sym}^2\Lambda^2H)(-1)$. This will be done in Section \ref{hyperelliptic J homo and Dehn twists} by computing the bracket of the images of  two commuting Dehn twists under the hyperelliptic Johnson homomorphism. Another important fact (Proposition \ref{even weight}) for computing the sections of $\Gr^W_\bullet df_\ast/W_{-5}$ for $f:\cC_{\cH_{g,n}}\to\cH_{g,n/k}$ is that $H_1(\u_T)$ with $T=\cH_{g/k}$ is pure of weight $-2$. This is a consequence of the fact that each component of the extended hyperelliptic Torelli space is simply connected.  \\
\indent Secondly, the relation between the sections of $\Gr^W_\bullet df_\ast/W_{-3}$ and the universal curve $f$ is established by the isomorphism 
$$\Hom_R(H_1(\u_T), H)\cong H^1_\et(T, \H).$$
Each tautological section $x_j$ of $f$ yields a characteristic class $\kappa_j$ in $H^1_\et(T, \H)$ and the class $\kappa_j/2g-2$ corresponds to the $j$th projection 
$$H_1(\u_T)\cong \Lambda^3_0H\oplus \bigoplus_{j=1}^nH_j\to H_j, \hspace{.3in}(v; u_1,\ldots,u_n)\mapsto u_j,$$
which determines the $j$th section of $\Gr^W_\bullet df_\ast/W_{-3}$. In the hyperelliptic case, we have the pullback classes in $H^1_\et(\cH_{g,n},\H)$ and  we will show that the $2n$ classes $\pm\kappa_1,\ldots,\pm\kappa_n$ determine the sections of the four-step graded Lie algebra map
$$\Gr^W_\bullet df_\ast/W_{-5}:\Gr^W_\bullet\u_C/W_{-5}\to\Gr^W_\bullet\u_T/W_{-5}$$ associated to the universal hyperelliptic curve. 
The key fact for establishing the bijection between the sections of the universal curve $f$ and that of $\Gr^W_\bullet df_\ast/W_{-3}$ is that two distinct sections of $f$ occupying the same class in $H^1_\et(\M_{g,n},\H)$ must be disjoint and hence there is a morphism 
$\M_{g,n}\to\M_{g,2}$  that induces a homomorphism
$$\Gr^W_\bullet\u_{\M_{g,n}}/W_{-3}\to\Gr^W_\bullet\u_{\M_{g,2}}/W_{-3},$$ but this homomorphism
is not compatible with their Lie algebra structures. In the proof of this fact \cite[Lemma 13.1]{hain2}, the existence of the representation $\Lambda^3_0H$ in $H_1(\u_T)$ plays an essential role, but in our case, there is no such representation. This issue is overcome by using the relations on $\Gr^W_\bullet\p_{g,n}$, which is determined in \cite[Thm.~12.6]{hain0}. \\
\indent Finally, we will consider the commutative diagram

$$\xymatrix{
	1\ar[r]&(\cP/W_{-N})(\Ql)\ar[r]\ar[d]&G_N\ar[r]\ar[d]&\pi_1(T)\ar[r]\ar[d]&1\\
	1\ar[r]&(\cP/W_{-N})(\Ql)\ar[r]      &(\cG/W_{-N}\cP)(\Ql)\ar[r]&\cG_T(\Ql)\ar[r]&1,
}
$$ 
where the right-hand square is the pullback square.  Under the finite-dimensional condition on $H^1(\g_T, \Gr^W_{-m}\p)$ with $1\leq m < N$, the set $H^1_\nab(\cG_T, \cP/W_{-N})(\Ql)$ of the  $(\cP/W_{-N})(\Ql)$-conjugacy classes of the sections of $(\cG/W_{-N}\cP)(\Ql)\to \cG_T(\Ql)$  is in bijection with the set $H^1_\nab(\pi_1(T), \cP/W_{-N})(\Ql)$ of the  $(\cP/W_{-N})(\Ql)$-conjugacy classes of the sections of $G_N\to \pi_1(T)$. For $f:C\to T$  the universal curve of genus $g\geq 3$, there is a bijection between $H^1_\nab(\cG_T, \cP/W_{-N})(\Ql)$ and the $R$-invariant Lie algebra sections of $\Gr^W_\bullet\g_C/W_{-N}\to\Gr^W_\bullet\g_T/W_{-N}$ with $N=3$. In the hyperelliptic case, the same statement holds for $N=5$. In order to compute for all $N$, we will use the ``exact" sequence
$$H^1(\g_T, \Gr^W_{-N}\p)\curvearrowright H^1_\nab(\cG_T, \cP/W_{-N-1})\overset{p}\to H^1_\nab(\cG_T, \cP/W_{-N})\overset{\delta}\to H^2(\g_T, \Gr^W_{-N}\p).$$
The interpretation of this sequence as an exact sequence comes from the property that for each $N$, the restriction of the projection $p$ to $\delta^{-1}(0)$ is a principal $H^1(\g_T, \Gr^W_{-N}\p)$-bundle. Using this exact sequence, we will compute the nonabelian cohomology $H^1_\nab(\cG_T, \cP/W_{-N})$ for all $N$ inductively, and then we will obtain our result by observing 
$$H^1_\nab(\cG_T, \cP)=\varprojlim_N H^1_\nab(\cG_T, \cP/W_{-N}).$$

\section{Hyperelliptic mapping class groups}
Suppose that $\Sigma_g$ is a compact oriented surface of genus $g$. Let $P$ be a set consisting of $n$ points on $\Sigma_g$.  Let $g,n$ be nonnegative integers such that $2g-2+n>0$. The mapping class group $\G_{g,n}$ of the surface $\Sigma_g$ is the group of isotopy classes of orientation preserving diffeomorphisms of $\Sigma_g$ that fix $P$ pointwise:
$$\G_{g,n}=\pi_0 \mathrm{Diff}^+(\Sigma_g, P).$$ The group $\G_{g,n}$ is independent of the choice of $\Sigma_g$ and $P$. When $n=0$, we will denote $\G_{g,0}$ by $\G_g$.  Fix a hyperelliptic involution $\sigma: \Sigma_g\to \Sigma_g$, which is an orientation preserving  diffeomorphism of $S$ of order 2 with $2g+2$ fixed points. The class $[\sigma]$ is unique up to conjugation in $\G_g$. The hyperelliptic mapping class group $\Delta_g$ is defined to be the centralizer of $[\sigma]$ in $\G_g$:
$$\Delta_g:=\{\phi\in \G_g|\phi[\sigma]\phi^{-1}=[\sigma]\}.$$
Define $\Delta_{g,n}$ to be the fiber product
$$\Delta_{g,n}:=\G_{g,n}\times_{\G_g}\Delta_g,$$
where $\G_{g,n}\to\G_g$ is the natural projection obtained by forgetting the $n$ marked points.
Set $H_\Z=H_1(\Sigma_g, \Z)$. The group $\Delta_g$ acts on $H_\Z$, which gives a natural representation $$\rho:\Delta_g\to\Sp(H_\Z).$$ Denote the image of $\rho$ by $G_g$ and the kernel by $T\Delta_g$. The group $G_g$ is a finite-index subgroup of $\Sp(H_\Z)$ by a result of A'Campo \cite{Acamp}. Taking mod $r$ induces the surjection $\Sp(H_\Z)\to \Sp(H_{\Z/r\Z})$ and define the level $r$ subgroup $\Delta_g[r]$ to be the kernel of the composition 
$$\Delta_g\to \Sp(H_\Z)\to \Sp(H_{\Z/r\Z}).$$
Then define the level $r$ subgroup $\Delta_{g,n}[r]$ of $\Delta_{g,n}$ to be the fiber product 
$$\Delta_{g,n}[r]:=\Delta_{g,n}\times_{\Delta_g}\Delta_{g}[r].$$ 
If $r\geq 3$, then the group $\Delta_{g,n}[r]$ is torsion free. We will denote the image of $\Delta_g[r]$ in $\Sp(H_\Z)$ by $G_g[r]$. It is the kernel of the map $G_g\to \Sp(H_{\Z/r\Z})$. For each $r\geq 1$, there is an exact sequence
$$1\to T\Delta_g\to \Delta_g[r]\to G_g[r]\to 1.$$

Define also $\Delta_{g,(1)}$ to be the subgroup of $\Delta_g$ that fixes a particular Weierstrass point. More precisely, define $\Delta_{g,(1)}$ as follows. Denote the set of $2g+2$ fixed points of $\sigma$ by $W$. The hyperelliptic mapping class group $\Delta_g$ acts on the set $W$ and it is easily seen to be surjective.  
Fixing a Weierstrass point $p$, we have the stabilizer subgroup $\Aut(W)_p$ of $\Aut(W)$.  Define $\Delta_{g,(1)}$ to be the pullback of $\Aut(W)_p$ along the map $\Delta_g\to \Aut(W)$: there is a commutative diagram
$$
\xymatrix{\Delta_{g,(1)}\ar[r]\ar[d]&\Delta_g\ar[d]\\
	\Aut(W)_p\ar[r]&\Aut(W).
}
$$

\section{Moduli stacks of smooth hyperelliptic curves}
\subsection{Analytic approach}
Let $X_{g,n}$ be the Teichm$\ddot{\text{u}}$ller space of marked, $n$-pointed, compact complex curves of genus $g$. As a set, the space $X_{g,n}$ consists of the isotopy classes of markings $[f:(\Sigma_g,P)\to (C, x_1,\ldots, x_n)]$. The mapping class group $\G_{g,n}$ acts on $X_{g,n}$ by $[\phi]:[f]\mapsto [f\circ \phi]$. It is known that this action is properly discontinuous and virtually free. It follows from results in \cite{ear} that the locus $X^{\hyp,\sigma}_g$ fixed by $\sigma$ is a complex submanifold of $X_g$ that is biholomorphic to $X_{0,2g+2}$ an hence contractible of dimension $2g-1$. The subgroup of $\G_g$ that acts on $X^{\hyp,\sigma}_g$ is exactly $\Delta_g$.  Since every hyperelliptic involution is conjugate to each other in $\G_g$,  the hyperelliptic locus  $X_g^\hyp$ in $X_g$ is the disjoint union 
$$
X^\hyp_g=\bigcup_{\phi\in \G_g/\Delta_g}\phi\cdot X^{\hyp,\sigma}_g=\bigcup_{\phi\in \G_g/\Delta_g} X^{\hyp,\phi\sigma\phi^{-1}}_g.$$
As an orbifold, the moduli stack of smooth projective hyperelliptic curves of genus $g$ is the orbifold quotient $$\cH_g=\Delta_g\backslash X_g^{\hyp,\sigma}.$$

\subsection{Algebraic approach}
Assume that $k$ is a field of characteristic zero. In this paper, all curves are connected, proper and smooth. Denote the moduli functor of hyperelliptic curves over $k$ by $\cH_{g/k}$. It is a contravariant functor from the category of $k$-schemes to the category of sets:
$$\cH_{g/k}: Sch_{/k}\to Set$$
which assigns to each $k$-scheme $T$ the set of isomorphism classes of families of hyperelliptic curves of genus $g$:
$$\cH_{g/k}(T)=\{C\to T\text{ family of hyperelliptic curves of genus $g$}\}/\cong.$$ It was shown by Arsie and Vistoli \cite{AV} that it is an irreducible smooth Deligne-Mumford stack of finite type over $k$ of dimension $2g-1$.  When $k\subset \C$,  the corresponding analytic stack is isomorphic to $\cH_{g}$. Also it is a closed substack of the moduli stack $\M_{g/k}$ of smooth proper curves of genus $g$ over $k$. Define $\cH_{g,n/k}$ to be the fiber product
$$\cH_{g,n/k}=\M_{g,n/k}\times_{\M_{g/k}}\cH_{g/k}.$$ Denote the universal curve over $\cH_{g,n/k}$ by $\cC_{\cH_{g,n}/k}\to \cH_{g,n/k}$. It is the pullback of $f:\cC_{g,n/k}\to \M_{g,n/k}$ to $\cH_{g,n/k}$, which we will also denote by $f$. 

\subsection{Curves of compact type}
Denote by $\overline{\M}_{g,n/k}$ the Deligne-Mumford compactification  of the moduli stack $\M_{g,n/k}$ of proper smooth $n$-pointed curves of genus $g$ over $k$. A projective nodal curve $C$ is said to be of compact type if its dual graph is a tree or equivalently $\Pic^0C$ is an abelian variety. The moduli stack $\M^{c}_{g,n}$ of $n$-pointed curves of genus $g$ of compact type is the complement $\overline{\M}_{g,n/k}-\delta_0$ of the divisor $\delta_0$, whose generic point corresponds to an irreducible $n$-pointed curve with one node. The moduli stack $\cH_{g,n/k}^{c}$ of $n$-pointed hyperelliptic curves of genus $g$ of compact type is the closure of $\cH_{g,n/k}$ in $\M_{g,n/k}^{c}$.   The moduli stack of $n$-pointed curves of genus $g$ with an abelian level $r$ will be denoted by $\M_{g,n/k}[r]$. It is given by fixing an isomorphism $\phi:\mathrm{R}^1f_\ast\Z/r\Z\cong (\Z/r\Z)^{2g}$ that respects their respective symplectic structures. 
For a standard reference for level structure on $\M_{g/k}$, see \cite{DM}.  In order to extend the abelian level to $\M^c_{g,n/k}$, note that the Torelli map $\M_{g,n/k}\to \A_{g/k}$ extends to $\M^c_{g,n/k}$, where $\A_{g/k}$ is the moduli stack of principally polarized abelian varieties of dimension $g$ over $k$. A similar construction gives an abelian level $r$ structure $\A_{g/k}[r]$, which is a finite \'etale cover of $\A_{g/k}$. The pullback of $\A_{g/k}[r]$ to $\M^c_{g,n/k}$ will be denoted by $\M^c_{g,n/k}[r]$. The level $r$ structure $\cH_{g,n/k}[r]$ denotes a particular connected component of the pullback of $\M_{g,n/k}[r]$ to $\cH_{g,n/k}$ and the closure of  $\cH_{g,n/k}[r]$ in $\M^c_{g,n/k}[r]$ will be denoted by $\cH^c_{g,n/k}[r]$. In this paper, we always assume that the base field $k$ contains the $r$th root of unity $\mu_r(\bar k)$ so that $\cH^c_{g,n/k}[r]$ is a geometrically connected smooth stack over $k$.  When $r\geq 3$, it is a smooth quasi-projective variety over $k$. \\	
\indent The Torelli space denoted by $\T_g$ is the quotient space $T_g\backslash X_g$, where $T_g$ is the Torelli group of genus $g$ that is the kernel of the natural representation $\G_g\to \Sp(H_\Z)$. The deformation theory of stable curves implies that there is the extended Torelli space $\T^{c}_g$ that contains $\T_g$ as an open dense subset . It is a complex manifold that parametrize curves of compact type with a homology framing. The hyperelliptic Torelli space denoted by $\T^\hyp_g$ is the image of $X^\hyp_g$ in $\T_g$, which is a finite disjoint union of mutually isomorphic irreducible components, each of which is isomorphic to the quotient $T\Delta_g\backslash X^{\hyp,\sigma}_g$. Denote the quotient $T\Delta_g\backslash X^{\hyp,\sigma}_g$ by $\cH_g[0]$. The closure of $\cH_g[0]$ in $\T^c_g$ will be denoted by $\cH^c_g[0]$. It was proved by Brendle, Margalit, and Putman in \cite[Thm.~B]{BMP} that $\cH^c_g[0]$ is simply connected. This fact plays a key role in the proof of Proposition \ref{even weight}. For each $r\geq 1$, the action of $G_g[r]$ on $\cH_g[0]$ extends to $\cH^c_g[0]$, and there are isomorphisms  
$$\cH_g[r]\cong G_g[r]\backslash\cH_g[0] \text{ and } \cH^c_g[r]\cong G_g[r]\backslash\cH_g^c[0]$$ as orbifolds. 

\begin{variant}
	Define $\cH_{g,(1)}$ to be the orbifold quotient $\Delta_{g,(1)}\backslash X_{g}^{\hyp,\sigma}$. The stack $\cH_g[2]$ is a finite Galois cover of $\cH_g$ that factors through the cover $\cH_{g,(1)}$:
$$	\cH_{g}[2]\to\cH_{g,(1)}\to\cH_{g}.
	$$
	
\end{variant}

\begin{variant}
	Denote the $n$th power of the universal curve $\M_{g/k}$ by $\cC_{g/k}^n$ and the $n$th power of the universal hyperelliptic curve over $\cH_{g/k}$ by $\cC_{\cH_g/k}^n$. By convention, we set $\cC^0_{g/k}=\M_{g/k}$ and there is an isomorphism  $\cC^1_{g/k}\cong \M_{g,1/k}$.
	Note that $\M_{g,n/k}$ is an open substack of $\cC_{g/k}^n$ and that $\cH_{g,n/k}$ is an open substack of $\cC_{\cH_g/k}^n$. We have the following fiber product diagram
	$$
	\xymatrix{
		\cH_{g,n/k}\ar[r]\ar[d]&\cC_{\cH_g/k}^n\ar[r]\ar[d]&\cH_{g/k}\ar[d]\\
		\M_{g,n/k}\ar[r]       &\cC_{g/k}^n\ar[r]          &\M_{g/k}.
	}
	$$
	Forgetting the first component of $\cC^{n+1}_{g/k}$  gives a curve $\pi:\cC^{n+1}_{g/k}\to \cC^{n}_{g/k}$. Pulling back $\pi$ to $\M_{g,n/k}$, we obtain the universal curve $\cC_{g,n/k}\to \M_{g,n/k}$, which pulls back to the universal hyperelliptic curve $\cC_{\cH_{g,n}/k}\to\cH_{g,n/k}$.
\end{variant}

\section{Fundamental groups of $\pi_1(\cH_{g,n/k})$ and their monodromy representations}
\subsection{Monodromy action associated to the universal hyperelliptic curves}
Fix an algebraic closure $\bar k$ of $k$. By the comparison theorem \cite{noo2}, we have a unique conjugacy class of isomorphisms 
$$\pi_1(\cH_{g,n/\bar k})\cong \pi_1^\orb(\cH_{g,n})\widehat{~}\cong (\Delta_{g,n})\widehat{~}$$
between the \'etale fundamental group of $\cH_{g,n/\bar k}$ and the profinite completion of its orbifold fundamental group.
 Let $\bar \eta:\Spec \Omega\to \cH_{g,n/\bar k}$ be a geometric point of $\cH_{g,n/\bar k}$. We also regard $\bar \eta$ as a geometric point of $\cH_{g,n/k}$ via the natural morphism $\cH_{g,n/\bar k}\to\cH_{g,n/k}$. The fundamental group  $\pi_1(\cH_{g,n/k}, \bar \eta)$ sits in the exact sequence of \'etale fundamental groups
$$1\to \pi_1(\cH_{g,n/\bar k}, \bar \eta)\to \pi_1(\cH_{g,n/k}, \bar \eta)\to G_k\to 1. $$\\
Let $\ell$ be a prime number. For $A=\Zl$ or $A=\Ql$, set $\H_{A}=\mathrm{R}^1f_\ast A(1)$, and let $H_A=(\H_A)_{\bar \eta}=H^1_\et(C_{\bar\eta}, A(1))$, where $C_{\bar\eta}$ is the fiber of the universal hyperelliptic curve $f: \cC_{\cH_{g,n}/k}\to\cH_{g,n/k}$ over the geometric point $\bar\eta$.
The cohomology group $H_A$ is equipped with a nondegenerate alternating form $\theta: H_A\otimes H_A\to A(1)$.  Associated to the local system $\H_A$, there is a natural monodromy representation 
$$\rho_\etabar:\pi_1(\cH_{g,n/ k}, \bar\eta)\to \GSp(H_\Zl).$$ The restriction to the geometric part, $\rho^\geom_\etabar:\pi_1(\cH_{g,n/\bar k}, \bar{\eta})\to \Sp(H_\Zl)$, agrees with the profinite completion of the action of $\Delta_g$ on $H_1(\Sigma_g,\Z)\otimes\Zl$. It follows from the result of A'Campo \cite{Acamp} that the image of $\rho^\geom_\etabar$ in $\Sp(H_\Zl)$ contains a finite-index subgroup and hence that $\rho^\geom_\etabar$ has a Zariski-dense image in $\Sp(H_\Ql)$. The monodromy representations $\rho_\etabar$ and $\rho^\geom_\etabar$ fit in the commutative diagram
$$\xymatrix{
	1\ar[r]&\pi_1(\cH_{g,n/\bar k}, \bar\eta)\ar[r]\ar[d]^{\rho^\geom_\etabar}&\pi_1(\cH_{g,n/k}, \bar\eta)\ar[r]\ar[d]^{\rho_\etabar}&G_k\ar[r]\ar[d]^{\chi_\ell}&1\\
	1\ar[r]&\Sp(H_{\Ql})\ar[r]&\GSp(H_\Ql)\ar[r]&\Gm_{/\Ql}\ar[r]&1,
}
$$
where $\chi_\ell$ is the $\ell$-adic cyclotomic character. If the image of $\chi_\ell$ is infinite, then the Zariski density of $\rho^\geom_\etabar$ implies that $\rho_\etabar$ also has a Zariski-dense image. \\
\subsection{$\GSp$-representations} Let $H$ be a $\Q$-vector space of dimension $2g$ equipped with a nondegenerate alternating form $\theta: \Lambda^2H\to \Q$. . The group of symplectic similitudes of $H$ is
$$\GSp(H)=\left\{\phi\in \GL(H)|\text{there exists a unit $u_\phi$ in $\Q$ such that }\phi^\ast\theta=u_\phi\theta \right\}.$$ The association $\phi\mapsto u_\phi$ defines a surjective homomorphism $\tau:\GSp(H)\to \Gm_{/\Q}$ and the symplectic group $\Sp(H)$ is the kernel of $\tau$: the sequence
$$1\to\Sp(H)\to \GSp(H)\overset{\tau}\to \Gm_{/\Q}\to1$$ is exact. We may consider $\theta$ as a $\GSp(H)$-invariant homomorphism
$$\theta:\Lambda^2H\to \Q(1),$$
where $\GSp(H)$ acts on $\Q(1)$ via the homomorphism $\tau$.  \\
\indent Each finite dimensional irreducible representation of $\Sp(H)$ is given by a partition $\lambda$ of a nonnegative integer $n$ into $l\leq g$ parts
$$n=\lambda_1+\lambda_2+\cdots+\lambda_l \text{ with } \lambda_1\geq\lambda_2\geq\cdots\geq\lambda_l>0.$$
 We will denote the isomorphism class of the irreducible representation corresponding to the partition $\lambda$ by $V_{\lambda}$. For a standard reference, see  \cite{ful}. Denote by $\Q(m)$ the one-dimensional $\GSp(H)$-representation on which $\GSp(H)$ acts via the $m$th power of $\tau$. Then the isomorphism classes of irreducible representations of $\GSp(H)$ are given by $V_\lambda(m):=V_\lambda\otimes \Q(m)$ with $m\in \Z$. We will call tensoring with $\Q(m)$ a Tate twist as in Hodge Theory.  \\
\indent In order to set weights on the $\GSp(H)$-representations, we will define the central cocharacter 
$$\omega: \Gm\to \GSp(H), \hspace{.3in} a\mapsto a^{-1}\id_H.$$ In this choice of $\omega$, the $\Gm$-representation weights coincide with the weights determined by Frobenius and by Hodge theory. The representation $V_\lambda$ has weight equal to $-|\lambda|$, where $|\lambda|=\lambda_1+\cdots+\lambda_l$. The one-dimensional representation $\Q(m)$ has weight equal to $-2m$. Therefore, the representation $V_\lambda (m)$ has weight $|V_\lambda(m)|=-|\lambda|-2m$. \\

\subsection{General symplectic local systems} The fact that $\GSp(H)$ is split over $\Q$ implies that the representations of $\GSp(H_\Ql)$ are obtained from those of $\GSp(H)$ by tensoring with $\Ql$, and the isomorphism classes of irreducible representations of $\GSp(H_\Ql)$ will be also denoted by $V_\lambda(m)$. \\
 \indent For each irreducible  $\GSp(H_\Ql)$-representation $V_\lambda(m)$, there is a corresponding local system $\V_\lambda(m)$ over $\cH_{g,n/k}$ via $\rho_\etabar$. That the kernel of $\pi_1(\M_{g,n/\bar k},\etabar)\to \pi_1(\M_{g,n/\bar k}^c,\etabar)$ acts trivially on $H_\Ql$ implies that for $r\geq 1$, the monodromy representation $\rho_\etabar$ factors through $\pi_1(\cH^c_{g,n/k}, \etabar)$, i.e., there is a monodromy representation $\pi_1(\cH^c_{g,n/k}[r],\etabar)\to \GSp(H_\Ql)$. Thus the representation $V_\lambda(m)$ defines a local system $\V_\lambda(m)$  over $\cH^c_{g,n/k}[r]$ that extends the one over $\cH_{g,n/k}[r]$. 
\section{The hyperelliptic Johnson homomorphisms and Dehn twists} \label{hyperelliptic J homo and Dehn twists}

In this section, we briefly introduce the hyperelliptic analogue of the Johnson homomorphism. 
First we recall the Johnson homomorphism. Let $\Pi$ be the topological fundamental group $\pi_1^\top(\Sigma_g, x)$  of the surface $\Sigma_g$.  Denote the lower central series of $\Pi$ by $L_\bullet\Pi$: 
$$L_1\Pi=\Pi\text{ and }L_k\Pi=[L_{k-1}\Pi, \Pi]\text{ for }k\geq 2.$$
Let  $\Gr^L_m\Pi:=L_m\Pi/L_{m+1}\Pi$ for each $m\geq 1$.  Each graded quotient $\Gr^L_m\Pi$ is a torsion-free $\Sp(H_1(\Sigma,\Z))$-module of finite rank. 
Let $H=H_1(\Sigma_g, \Z)$ equipped with the standard symplectic basis $a_1, b_1, \ldots, a_g, b_g$ and the intersection paring $\theta: H\otimes H\to \Z$. The form $\theta$ is a nondegenerate skew-symmetric bilinear form. Having fixed a symplectic basis for $H$, we may identify $\Sp(H_1(\Sigma_g,\Z))$ with $\Sp_g(\Z)$, the group of $2g\times 2g$ symplectic matrices with entries in $\Z$.  We also regard $\theta$ as an element of $\Lambda^2H$ via the isomorphism $H\cong H^\ast$ induced by $\theta$. Then $\theta=\sum_{i=1}^g a_i\wedge b_i$.
 Note that the mapping class group $\G_{g,1}$ acts on $\Pi$ and there is an exact sequence of $\G_{g,1}$-groups
$$1\to \Gr^L_2\Pi\to \Pi/L_3\Pi\to \Gr^L_1\Pi\to 1$$ Identifying $\Gr_L^1\Pi\cong H$ and $\Gr_L^2\Pi\cong \Lambda^2H/\langle \theta\rangle$, this exact sequence can be rewritten as
$$1\to \Lambda^2H/\langle \theta\rangle \to \Pi/L^3\Pi\to H\to 1.$$  Recall that the Torelli group $T_{g,1}$ is the kernel of the natural representation $\G_{g,1}\to \Sp(H)$, which is known to be surjective.  Then the Johnson homomorphism 
$$\tau: T_{g,1}\to \Hom(H,\Lambda^2H/\langle \theta\rangle)$$
is defined as follows. For $u\in H$, let $\tilde{u}$ be a lift of $u$ in $\Pi/L_3\Pi$. For an element $\phi\in T_{g,1}$, the element $\phi(\tilde{u})\tilde{u}^{-1}$ lies in $\Lambda^2H/\langle \theta\rangle$, since $\phi$ acts trivially on $H$ by definition. The Johnson homomorphism $\tau$ is defined to be the map $\phi\mapsto (u\mapsto\phi(\tilde{u})\tilde{u}^{-1})$.  It can be easily checked that $\tau(\phi)$ is independent of the choice of the life $\tilde{u}$ of $u$.  Johnson showed in \cite{joh1, joh} that $\im\tau =\Lambda^3H$ for $g\geq 2$ and $H_1(T_{g,1},\Z)\otimes\Z[\frac{1}{2}]=\Lambda^3H\otimes \Z[\frac{1}{2}]$ for $g\geq3$.\\
\indent Recall that the hyperelliptic Torelli group $T\Delta_g$ is the kernel of the natural representation $\Delta_g\to \Sp(H)$ or equivalently defined to be the intersection of $T_g$ and $\Delta_g$ in $\G_g$. Since $T\Delta_g$ fixes each Weierstrass point, $T\Delta_g$ is a subgroup of $\Delta_{g,(1)}$. 
 A construction similar to one for the Johnson homomorphism  gives a $\Delta_{g,(1)}$-equivariant homomorphism 
$$ \omega: T\Delta_g\to \Hom(H,\Lambda^2H/\langle \theta\rangle).$$
However, this homomorphism is trivial, since the hyperelliptic involution $\sigma$ acts trivially on $T\Delta_g$, while it acts as $-\id$ on $\Hom(H,\Lambda^2H/\langle \theta\rangle)$. Thus this suggests that we consider the next graded quotient $\Gr^L_3\Pi$ and the exact sequence of $\Delta_{g,(1)}$-groups
$$1\to \Gr^L_3\Pi\to \Pi/L_4\Pi\to \Pi/L_3\Pi\to 1.$$
The triviality of the homomorphism $\omega$ implies that $T\Delta_g$ acts trivially on $\Pi/L_3\Pi$. Thus the same construction for the Johnson homomorphism gives a $\Delta_{g,(1)}$-equivariant homomorphism 
$$\tau^\hyp:T\Delta_g\to \Hom(\Pi/L^3\Pi, \Gr_L^3\Pi).$$
Since $\Gr_L^3\Pi$ is abelian and the abelianization of $\Pi/L_3\Pi$ is $\Gr^L_1\Pi=H$, we have 
$$\Hom(\Pi/L_3\Pi, \Gr^L_3\Pi)=\Hom(H, \Gr^L_3\Pi).$$
This $\Delta_{g,(1)}$-equivariant homomorphism 
$$\tau^\hyp: T\Delta_g\to \Hom(H, \Gr^L_3\Pi)$$
is defined to be the hyperelliptic Johnson homomorphism. 
Note that the restriction of $\tau^\hyp$ to the kernel $N$ of $\tau^\hyp$ induces a homomorphism
$$\tau^\hyp_2:N\to \Hom(H, \Gr_5^L\Pi).$$ More generally, if $N_j$ is the kernel of the homomorphism $\Delta_{g,(1)}\to \Aut(\Pi/L_{2j}\Pi)$, then there is a homomorphism
$$\tau^\hyp_j:N_j\to\Hom(H, \Gr_{2j+1}^L\Pi).$$
Note that $N_1=T\Delta_g$.
\subsection{The image of a Dehn twist under $\tau^\hyp$}\label{The image of Dehn twist}
In this section, set $H=H_1(\Sigma_g, \Q)$. For each $j=1,\ldots, g-1$, let $C_j$ be the separating simple closed curve in $\Sigma_g$ such that $C_j$ separates $\Sigma_g$ into two components $S'_j$ and $S''_j$ of genus $j$ and $g-j$, respectively. 
 Let $\theta'_j=\sum_{i=1}^ja_i\wedge b_i$ and $\theta''_j=\sum_{i=j+1}^ga_i\wedge b_i$. Note that $\theta=\theta'_j+\theta''_j$.  Let $\mathbb{L}(H)$ be the free Lie algebra generated by $H$. 
Denote by $\p$ the pronilpotent Lie algebra of the unipotent completion (Malcev completion see \cite{Qui}) of $\Pi$ over $\Q$.
  There is an isomorphism of graded Lie algebras $\Gr^L_\bullet\Pi\otimes \Q\cong \Gr^L_\bullet\p$. It follows from the main  result in \cite{DGMS}  that the associated graded Lie algebra $\Gr^L_\bullet\p$ has a  minimal presentation 
$$ \Gr^L_\bullet\p\cong \mathbb{L}(H)/\langle\theta\rangle.$$ 
Denote $\Gr^L_m\Pi\otimes \Q\cong \Gr^L_m\p$ by $\p(m)$.  In low degrees, we have 
$$\p(1)=H,\,\,\, \p(2)=\Lambda^2H/\langle\theta\rangle, \text{ and }\p(3)=\p(1)\otimes\p(2)/\Lambda^3\p(1).$$
We will express the vector in $\p(m)$ by the bracket of the vectors in $H$ of length $m$. 
Define a map $\phi:\mathrm{Sym}^2\Lambda^2H\to \Hom(H,\p(3))$ by
\begin{align*}
(u_1\wedge v_1)\cdot(u_2\wedge v_2)\mapsto &\Big( x\mapsto \theta(u_1, x)[v_1,[u_2,v_2]]-\theta(v_1, x)[u_1,[u_2,v_2]]\\
&+\theta(u_2, x)[v_2,[u_1,v_1]]-\theta(v_2,x)[u_2,[u_1,v_1]]\Big)
\end{align*}
It is easy to see that $\phi$ is an $\Sp(H)$-invariant homomorphism. 
\begin{lemma}\label{image of theta square}
	For $I\subset\{1,\ldots,g\}$, set $\theta_I=\sum_{i\in I}a_i\wedge b_i$, $H_I=\mathrm{Span}\{ a_i, b_i|i\in I\}$, and 
	$H^c_I=\mathrm{Span}\{ a_i, b_i|i\not \in I\}$. 
	Then we have
	\[\phi(\theta^2_I):x\mapsto \left\{
	\begin{array}{ll}
	\phantom{} 0 & \quad x \in H^c_I \\
	\phantom{}2[x,\theta_I] & \quad x\in H_I
	\end{array}
	\right.\]     
\end{lemma}
\begin{proof}A simple computation suffices.
\end{proof}
Let $\omega_j$ be the isotopy class of the Dehn twist around $C_j$. For simplicity, fix a Weierstrass point $p$ in $S'_1$. Note that each $\omega_j$ is an element in $T\Delta_g$.  In order to compute $\tau^\hyp(\omega_j)$, we need to compute the action of $\omega$ on $\pi_1(\Sigma_g, p)$. We do this by computing the action of $\omega_j$ on the standard generating set of $\pi_1(\Sigma_g, p)$ consisting of the classes $\gamma_{2i-1}$, $\gamma_{2i}$  based at $p$ for $i=1,\ldots, g$, where the homology classes of $\gamma_{2i-1}$ and $\gamma_{2i}$ are $a_i$ and $b_i$, respectively. 
\begin{proposition}\label{image of a dehn twist}\label{hyp j action}
	With notation as above, we have
	$$\tau^\hyp(\omega_j)=\frac{1}{2}\phi((\theta''_j)^2):x\mapsto \left\{
	\begin{array}{ll}
	\phantom{} 0 & \quad x \in H^c_I \\
	\phantom{}[x,\theta''_j] & \quad x\in H_I
	\end{array}
	\right.$$
	where $I=\{j+1,\ldots,g\}$.
\end{proposition}
\begin{proof}
	Since  the classes  $\gamma_{i}$ for $i\leq 2j$ do not intersect with $C_j$, we have $\omega_j(\gamma_i)=\gamma_i$ for $i\leq 2j$. Now, fix a point $y_j$ on $C_j$ and a path $\rho_j$ from $p$ to $y_j$. Considering $C_j$ as a loop based at $y_j$, the composition $\rho_jC_j\rho_j^{-1}$ is a loop based at $p$. Note that the homotopy class of $\rho_jC_j\rho_j^{-1}$ is equal to the class
	$$\left([\gamma_1,\gamma_2]\cdots[\gamma_{2j-1},\gamma_{2j}]\right)^{-1}\in \pi_1(\Sigma_g,p).$$
	The curve $C_j$ is traversed so that
	for $i>2j$, we have 
	\begin{align*}
	\omega_j(\gamma_i)&=\left([\gamma_1,\gamma_2]\cdots[\gamma_{2j-1},\gamma_{2j}]\right)\gamma_i \left([\gamma_1,\gamma_2]\cdots[\gamma_{2j-1},\gamma_{2j}]\right)^{-1}\\
	&=\left[\prod_{k=1}^{j}[\gamma_{2k-1},\gamma_{2k}],\gamma_i\right]\gamma_i,
	\end{align*}
	which shows that 
	$\omega_j(\gamma_i)\gamma_i^{-1}\in L_3\Pi.$
	Reducing mod $L_4\Pi$, we obtain the class
	$$\left[\sum_{k=1}^j[a_k,b_k], \tilde{\gamma}_i\right]\in \Gr^3_L\Pi,$$
	where $\tilde{\gamma}_i$ is the homology class of $\gamma_i$. 
	Since $\sum_{k=1}^j[a_k,b_{k}]=\theta'_j$ and $\theta'_j+\theta''_j=\theta$, we can express
	$$\left[\sum_{k=1}^j[a_k,b_k], \tilde{\gamma}_i\right]=[\theta'_j,\tilde{\gamma}_i]=[\tilde{\gamma}_i,\theta''_j]\text{ in }\p(3).$$
	The $\tilde{\gamma}_i$ with $i>2j$ form a basis for $H_I$, and thus we have
	$$\tau^\hyp(\omega_j)(x)=[x, \theta''_j]\,\,\,\,\,\text{ if }x\in H_I$$. It is clear that we have 
	$$\tau^\hyp(\omega_j)(x)=0\,\,\,\,\,\text{ if }x\in H^c_I.$$
\end{proof}
\subsection{The derivation Lie algebra $\Der \p $}
Here, we will quickly review the derivation Lie algebra $\Der\p$ and the $\Sp(H)$-representations appearing in low degree terms. 
The following result will be used in our computations. It is computed using the complex of chains $\Lambda^\bullet\Gr^L_\bullet\p$ \cite[Cor.~8.3]{hain0}.
\begin{proposition}[\cite{hain0}]\label{presentation}
	For all $g\geq 3$, the irreducible decomposition of $\p(m)$ in the category of $\Sp(H)$-modules when $1\leq m\leq 4$ is given by
	\begin{align*}
	\p(1)&=V_{[1]};\\
	\p(2)&=V_{[1^2]};\\
     \p(3)&=V_{[2+1]};\\
	\p(4)&=V_{[2+1^2]}+V_{[2]}+V_{[3+1]}.\\
	\end{align*}
	Furthermore, for all $g\geq 3$, $\p(5)$ contains a copy of $V_{[3+1^2]}$ in its decomposition. 
\end{proposition}
\begin{remark}
	The above computation was done by the compute program  LiE developed at University of Amsterdam. 
\end{remark}
Denote the derivation Lie algebra of $\p$ by $\Der\p$. It is a graded Lie algebra:
$$
\Der\p=\bigoplus_n\Der_n\p,
$$
where $\Der_n\p$ is an $\Sp(H)$-submodule of $\Hom(\p(1),\p(n+1))$. More precisely, there is an $\Sp(H)$-invariant surjective homomorphism
$$p_n: \Hom(\p(1),\p(n+1))\cong\p(1)^\ast\otimes\p(n+1)\to \p(n+2)$$ 
that takes $\phi$ to the image of $\phi(\theta)$ in $\p(n+2)$. The homomorphism $\phi$ induces a derivation of $\p$ if and only if it maps to zero in $\p(n+2)$, i.e., $\Der_n\p=\ker p_n$. 
Together with the decomposition for each $\p(m)$, the following result can be easily computed.
\begin{corollary}[\cite{hain0}]
	For all $g\geq 3$, we have
	\begin{align*}
	\Der_1\p&=V_{[1^3]}+V_{[1]};\\
	\Der_2\p&=V_{[2^2]}+V_{[1^2]};\\
	\Der_3\p&=V_{[3+1^2]}+V_{[2+1]}+V_{[3]}.\\
	\end{align*}
\end{corollary}
It was shown in \cite{ak} that the graded Lie algebra $\Gr^L_\bullet\p$ has the trivial center, and hence we may regard $\Gr^L_\bullet\p$ as a graded Lie ideal of $\Der\p$ via the adjoint action. Denote the quotient $\Der\p/\im(\Gr^L_\bullet\p)$ by $\Out\Der\p$.
\begin{corollary}[\cite{hain0}]
	For all $g\geq 3$, we have
	\begin{align*}
	\Out\Der_1\p&=V_{[1^3]};\\
	\Out\Der_2\p&=V_{[2^2]};\\
	\Out\Der_3\p&=V_{[3+1^2]}+V_{[3]}.
	\end{align*}
\end{corollary}
The following lemma shows that $\tau^\hyp(\omega_j)$ induces a derivation of $\p$.
\begin{lemma} \label{image}
	For $I\subset \{1,\ldots,g\}$, $\phi(\theta_I^2)$ lies in $\Der_2\p$, 
	where  $\theta_I=\sum_{i\in I}a_i\wedge b_i$. 
\end{lemma}
\begin{proof}It will suffice to check that the element $\phi(\theta^2_I)$ as an element of $\Der\L(H)$ maps the polarization $\theta$ to zero. This can be checked through an easy computation using Lemma \ref{image of theta square}.
\end{proof}
	
\subsection{The outer action of a commuting pair of Dehn twists}\label{The outer action of a commuting pair of Dehn twists}
In this section, we will show that there is a commuting pair of Dehn twists such that the bracket of the outer part of the image of each of the Dehn twist under the hyperelliptic Johnson homomorphism  is a nontrivial inner derivation in $\Der_4\p$.
\begin{proposition}Assume $g\geq3$.
	The map $\phi:\mathrm{Sym}^2\Lambda^2H\to \Hom(H,\p(3))$ defined in \ref{The image of Dehn twist} induces a homomorphism 
	$$\tilde{\phi}: \mathrm{Sym}^2\Lambda^2H/(\Lambda^4H\oplus \Q)\to \Hom(H,\p(3)).$$
\end{proposition}
\begin{proof}
	First observe that the representation $\Lambda^4H$ sits inside $\mathrm{Sym}^2\Lambda^2H$ via the map
	$$v_1\wedge v_2\wedge v_3\wedge v_4\mapsto (v_1\wedge v_2)(v_3\wedge v_4)+(v_1\wedge v_3)(v_4\wedge v_2)+(v_1\wedge v_4)(v_2\wedge v_3).$$ The corresponding projection is given by
	$$(u_1\wedge u_2)(u_3\wedge u_4)\mapsto u_1\wedge u_2\wedge u_3\wedge u_4.$$ A copy of the trivial representation $\Q$ sits inside $\mathrm{Sym}^2\Lambda^2H$ via the map
	$1\mapsto \theta^2.$ Simple computations show that $\phi(\Lambda^4H)=0$ and $\phi(\theta^2)=0$. Thus $\phi$ induces an $\Sp(H)$-homomorphism
	$$\tilde{\phi}:\mathrm{Sym}^2\Lambda^2H/(\Lambda^4H\oplus \Q)\to \Hom(H,\p(3)).$$
\end{proof}
\begin{remark}
	In fact, the map $\tilde{\phi}$ is an isomorphism onto $\Der_2\p$. 
\end{remark}
In order to separate the $V_{[2^2]}$ component from $\mathrm{Sym}^2\Lambda^2H$, we need a projection onto a copy of $V_{[1^2]}$ that is not contained in the submodule $\Lambda^4H\subset \mathrm{Sym}^2\Lambda^2H$. Note that a copy of $\Lambda^2H$ sits inside $\Lambda^4H$ via the map
$ u\wedge v\mapsto u\wedge v\wedge\theta.$
The other copy of $\Lambda^2H$ sits inside $\mathrm{Sym}^2\Lambda^2H$ via the map
$ u\wedge v\mapsto (u\wedge v)\cdot\theta.$
Define a  map $\pi:\mathrm{Sym}^2\Lambda^2H\to \Lambda^2 H$ by
$$(u_1\wedge v_1)(u_2\wedge v_2)\mapsto \theta(u_1,v_1)v_2\wedge u_2+\theta(v_2,u_2)u_1\wedge v_1 \hspace{2in}$$
$$\hspace{1in}+\frac{1}{2}\{\theta(u_1,v_2)v_1\wedge u_2+\theta(v_1,u_2)u_1\wedge v_2 +\theta(u_1,u_2)v_2\wedge v_1+\theta(v_2,v_1)u_1\wedge u_2\}.$$
One can easily check that $\pi$ is an $\Sp(H)$-homomorphism and vanishes on $\Lambda^4 H$ as a submodule of $\mathrm{Sym}^2\Lambda^2H$.
\begin{lemma}
	Consider the irreducible $\Sp(H)$-module $V_{[1^2]}$ as the  submodule of $\Lambda^2H$ given by the kernel of the map $u\wedge v\mapsto \theta(u,v)$. Then the composition 
	$$V_{[1^2]}\overset{\cdot\theta}\to \mathrm{Sym}^2\Lambda^2H\overset{\pi}\to \Lambda^2H$$ is given by multiplication by $-g-1$. Also the composition
	$$\Q\theta\overset{\cdot\theta}\to \mathrm{Sym}^2\Lambda^2H\overset{\pi}\to \Lambda^2H$$
	is given by multiplication by $-2g-1$. \qed
\end{lemma}
Define $\Sp(H)$-invariant projections $p_1:\Lambda^2H\to\Q\theta$ by $u\wedge v\mapsto \frac{\theta(u,v)}{g}\theta$ and $p_2:\Lambda^2H\to V_{[1^2]}$ by $u\wedge v\mapsto u\wedge v-\frac{\theta(u,v)}{g}\theta$. Define $p:\Lambda^2 H\to\mathrm{Sym}^2\Lambda^2H$ by 
\[u\wedge v\mapsto\left( \frac{1}{-2g-1}(\cdot\theta\circ p_1)+\frac{1}{-g-1}(\cdot\theta\circ p_2)\right)(u\wedge v).\]
\begin{corollary}
	We have $\pi\circ p=\id$ on $\Lambda^2H$. \qed
\end{corollary}
\begin{corollary}\label{decomp} For $g\geq 2$, as an $\Sp(H)$-module, we have a decomposition
$$	\mathrm{Sym}^2\Lambda^2H\cong \ker\pi\oplus\im\pi\cong \Lambda^4H\oplus V_{[2^2]}\oplus\Lambda^2H.$$

\end{corollary}
\begin{proof}
Since a copy of $\Lambda^4H$ is contained in $\ker \pi$, it will suffice to see that $\ker\pi/\Lambda^4H$ is isomorphic to $V_{[2^2]}$.  This follows from a basic representation theory of $\Sp(H)$.
\end{proof}
For each vector $(u_1\wedge v_1)(u_2\wedge v_2)\in \mathrm{Sym}^2\Lambda^2H$, we can express uniquely 
$$(u_1\wedge v_1)(u_2\wedge v_2)= \delta_1+\delta_2+\delta_3,$$
where $\delta_1$, $\delta_2$, and $\delta_3$ are vectors of $\Lambda^4H$, $V_{[2^2]}$, and $\Lambda^2H$, respectively. 
\begin{corollary}\label{projection formula 1}
	We have 
	$$\delta_1+\delta_2=(u_1\wedge v_1)(u_2\wedge v_2)-\{\theta(u_1,v_1)p(v_2\wedge u_2)+\theta(v_2,u_2)p(u_1\wedge v_1)\}$$
	$$-\frac{1}{2}\{\theta(u_1,v_2)p(v_1\wedge u_2)+\theta(v_1,u_2)p(u_1\wedge v_2)+\theta(u_1,u_2)p(v_2\wedge v_1)+\theta(v_2,v_1)p(u_1,u_2)\}.$$ \qed
\end{corollary}
Now, we consider a specific pair of commuting Dehn twists on $\Sigma_g$. Consider the separating simple closed curves $C_1$ and $C_{g-1}$. Fix a Weierstrass point $p$ in $S''_1\cap S'_{g-1}$.
Recall that the corresponding isotopy classes of the Dehn twists around $C_1$ and $C_{g-1}$ are denoted by $\omega_1$ and $\omega_{g-1}$, respectively. By Proposition \ref{image of a dehn twist}, we have
$$\tau^\hyp(\omega_1)=\frac{1}{2}\phi((a_1\wedge b_1)^2),$$
and
$$\tau^\hyp(\omega_{g-1})=\frac{1}{2}\phi((a_g\wedge b_g)^2).$$
\begin{remark}In order to apply Proposition \ref{image of a dehn twist}, one needs to adjust the result according to the the fixed Weierstrass point on $\Sigma_g$.
\end{remark}
Denote $2\tau^\hyp(\omega_1)$ and $2\tau^\hyp(\omega_{g-1})$ by $\omega$ and $\tilde{\omega}$. By Lemma \ref{image}, $\omega$ and $\tilde{\omega}$ lies in $\Der_2\p$. Since $\Der_2\p=V_{[2^2]}+V_{[1^2]}$, we can express $\omega$ and $\tilde{\omega}$ as
$$\omega=\xi_{[2^2]}+\xi_{[1^2]},$$
and 
$$\tilde{\omega}=\tilde{\xi}_{[2^2]}+\tilde{\xi}_{[1^2]},$$
where $\xi_{[2^2]}$, $\tilde{\xi}_{[2^2]}$ are vectors in $V_{[2^2]}$ and $\xi_{[1^2]}$, $\tilde{\xi}_{[1^2]}$ are vectors in $V_{[1^2]}$. 
\begin{theorem}\label{no zero outer part} With notation as above, if $g\geq 3$, then the vector $[\xi_{[2^2]}, \tilde{\xi}_{[2^2]}]$ is a nontrivial inner derivation in $\Der_4\p$.
\end{theorem}
\begin{proof}
	The following diagram 
	$$\xymatrix{
		T\Delta_g\ar[r]^{\tau^\hyp_1}\ar[d]_{[~,~]}&\Der_2\p\ar[d]^{[~,~]}\\
		N_2\ar[r]^{\tau^\hyp_2}&\Der_4\p
	}
	$$
	commutes, where the left-hand vertical map is the map taking the commutator of two elements and the right-hand vertical map is the bracket of the derivation Lie algebra $\Der\p$. Since the Dehn twists $\omega_1$ and $\omega_{g-1}$ commute in $T\Delta_g$ being disjoint from each other, it then follows that $[\omega, \tilde{\omega}]=0$. Thus we have
	\begin{align*}
	0&=[\omega,\tilde{\omega}]\\
	&=[\xi_{[2^2]}+\xi_{[1^2]},\tilde{\xi}_{[2^2]}+\tilde{\xi}_{[1^2]}]\\
	&=[\xi_{[2^2]},\tilde{\xi}_{[2^2]}]+[\xi_{[2^2]},\tilde{\xi}_{[1^2]}]+[\xi_{[1^2]},\tilde{\xi}_{[2^2]}]+[\xi_{[1^2]},\tilde{\xi}_{[1^2]}],\\
	\end{align*}
	which shows that the derivation $[\xi_{[2^2]},\tilde{\xi}_{[2^2]}]$ is inner in $\Der\p$, since $\p$ is a Lie ideal in $\Der\p$ via the adjoint action. It remains to show that this derivation is nontrivial. 
	By Corollary \ref{decomp}, we can express $(a_1\wedge b_1)^2$ and $(a_g\wedge b_g)^2$ as
	$$(a_1\wedge b_1)^2=\delta_1+\delta_2+\delta_3\text{ in }\Lambda^4H\oplus V_{[2^2]}\oplus\Lambda^2H$$
	and 
	$$(a_g\wedge b_g)^2=\tilde{\delta}_1+\tilde{\delta}_2+\tilde{\delta}_3\text{ in }\Lambda^4H\oplus V_{[2^2]}\oplus\Lambda^2H.$$
	Recall that the homomorphism $\phi:\mathrm{Sym}^2\Lambda^2H\to \Hom(H,\p(3))$ vanishes on $\Lambda^4H$. Thus we have
	$$\phi(\delta_2)=\xi_{[2^2]}\text{ and }\phi(\delta_3)=\xi_{[1^2]}$$
	and
	$$\phi(\tilde{\delta}_2)=\tilde{\xi}_{[2^2]}\text{ and }\phi(\tilde{\delta}_3)=\tilde{\xi}_{[1^2]}.$$
	Using Corollary \ref{projection formula 1}, we obtain
	$$\delta_1+\delta_2=(a_1\wedge b_1)^2-\frac{3}{g+1}(a_1\wedge b_1)\cdot\theta+\frac{3}{(g+1)(2g+1)}\theta^2$$
	and 
	$$\tilde{\delta}_1+\tilde{\delta}_2=(a_g\wedge b_g)^2-\frac{3}{g+1}(a_g\wedge b_g)\cdot\theta+\frac{3}{(g+1)(2g+1)}\theta^2.$$
	Since $\phi$ vanishes on $\Lambda^4H$ and $\Q\theta^2$, we then have
	$$[\xi_{[2^2]},\tilde{\xi}_{[2^2]}]=\left[\phi\left((a_1\wedge b_1)^2-\frac{3}{g+1}(a_1\wedge b_1)\cdot \theta\right), \phi\left((a_g\wedge b_g)^2-\frac{3}{g+1}(a_g\wedge b_g)\cdot \theta\right)\right].$$ 
	We evaluate this bracket on the vector $a_2\in H$:
	\begin{align*}
	[\xi_{[2^2]},\tilde{\xi}_{[2^2]}](a_2)=&[\phi((a_1\wedge b_1)^2), \phi((a_g\wedge b_g)^2)](a_2)\\
	                                       & -\frac{3}{g+1}[\phi((a_1\wedge b_1)^2), \phi((a_g\wedge b_g)\cdot \theta)](a_2)\\
	                                      &-\frac{3}{g+1}[\phi((a_1\wedge b_1)\cdot \theta),\phi((a_g\wedge b_g)^2) ](a_2)\\
	                                      &+\frac{9}{(g+1)^2}[\phi((a_1\wedge b_1)\cdot \theta),\phi((a_g\wedge b_g)\cdot \theta)](a_2)\\
	                                      =&\frac{9}{(g+1)^2}[\phi((a_1\wedge b_1)\cdot \theta),\phi((a_g\wedge b_g)\cdot \theta)](a_2)\\
	                                      =&\frac{9}{(g+1)^2}[\phi((a_1\wedge b_1)(a_2\wedge b_2)),\phi((a_g\wedge b_g)(a_1\wedge b_1))](a_2)\\
	                                      &+\frac{9}{(g+1)^2}[\phi((a_1\wedge b_1)(a_2\wedge b_2)),\phi((a_g\wedge b_g)(a_2\wedge b_2))](a_2)\\
                                          &+\frac{9}{(g+1)^2}[\phi((a_1\wedge b_1)(a_g\wedge b_g)),\phi((a_g\wedge b_g)(a_2\wedge b_2))](a_2)\\      
                                          =&\frac{9}{(g+1)^2}(-[[[[a_g,b_g],a_1],b_1],a_2]+[[a_1,[b_1,[a_g,b_g]]]],a_2]\\
                                          &+[[a_g,b_g],[[a_1,b_1],a_2]]-[[a_1,b_1],[[a_g,b_g],a_2]]\\
                                          &+[[[[a_1,b_1],a_g],b_g],a_2]-[[a_g,[b_g,[a_1,b_1]]],a_2] )\\          =&\frac{9}{(g+1)^2}(-[a_2,[[a_1,b_1],[a_g,b_g]]]+[a_2,[[a_1,b_1],[a_g,b_g]]]\\&-[a_2,[[a_1,b_1],[a_g,b_g]]])\\
                                          =&-\frac{9}{(g+1)^2}[a_2,[[a_1,b_1],[a_g,b_g]]]
	                                   	\end{align*}

	This computation can be also  diagrammatically  described (see \cite{mss} for the description). Thus we have
	$$[\xi_{[2^2]},\tilde{\xi}_{[2^2]}](a_2)=-\frac{9}{(g+1)^2}\Big[a_2,\big[[a_1,b_1],[a_g,b_g]\big]\Big]\in \p(5).$$ This vector can be mapped to the highest vector of the copy of  $V_{[3+1^2]}$ appearing in $\p(5)$ via the action of $\s\p$, where $\s\p$ is the Lie algebra of $\Sp(H)$. Hence it is nonzero in $\Gr^L_\bullet\p$. Thus the derivation $[\xi_{[2^2]},\tilde{\xi}_{[2^2]}]$ is a nonzero inner derivation in $\Der_4\p$. 
\end{proof}

\section{Relative and Weighted completions of hyperelliptic mapping class groups}
\subsection{Review of weighted completion of a profinite group}
We will begin this section by reviewing briefly the theory of weight completion of a profinite group. The detailed introduction  and the results stated here can be found in \cite{hain2,wei}. It is a variant of relative completion of a discrete group and it linearizes a profinite group. \\
\indent  Suppose that $\G$ is a profinite group, $R$ a reductive group over $\Ql$, $\omega:\Gm\to R $ is a central cocharacter, and $\rho:\G\to R(\Ql)$ a continuous representation whose image is Zariksi-dense. An extension $G$ of $R$ is said to be negatively weighted if it is an extension of $R$ by a unipotent $\Ql$-group $U$
such that $H_1(U)$ as a $\Gm$-representation via $\omega$ has only negative weights. The weighted completion $(\cG\to R, \tilde{\rho}:\G\to \cG(\Ql))$ of $\G$ with respect o $\rho$ and $\omega$ is the projective limit of the pairs of the form 
$$(\pi_G: G\to R, \rho_G:\G\to G(\Ql)),$$
where $G$ is a negatively weighted extension of $R$ and $\rho_G$ is a continuous Zariski-dense representation lifting $\rho$, i.e., $\pi_G(\Ql)\circ \rho_G=\rho$. \\
\indent The completion $\cG$ is a negatively weighted extension of $R$ by a prounipotent $\Ql$-group $\U$. The Levi's theorem implies that the extension 
$$1\to \U\to\cG\to R\to 1$$ splits and any two splittings are conjugate by an element of $\U$.  Fix a splitting $s:R\to \cG$. Denote the Lie algebras of $\cG$, $\U$, and $R$ by $\g$, $\u$, and $\r$, respectively. The adjoint action of $\cG$ on $\g$ induces natural weight filtrations $W_\bullet$ on these Lie algebras via $s\circ\omega $. It can be shown that the weight filtrations do not depend on the choice of a splitting $s$. In below, we summarize the key properties of the weight filtrations:
\begin{theorem} The natural weight filtration $W_\bullet$ satisfies the following properties.
	\begin{enumerate}\label{weighted property 1}
		\item $W_0\g=\g$, $W_{-1}\g=W_{-1}\u$, and $\Gr^W_0\g=\r$.
		\item the action of $\U$ on $\Gr^W_m\g$ is trivial and hence each graded piece $\Gr^W_m\g$ is an $R$-module.
		\item the functor $\Gr^W_\bullet$ in the category of  $\cG$-modules is exact.
	\end{enumerate}
\end{theorem}	
The following result shows that $H_1(\u)$ is determined by the cohomology of $\G$. Note that each $R$-representation $V$ is also a $\G$-representation via $\rho$. 
\begin{proposition}\label{generator iso} Let $V$ be a finite dimensional $R$-representation of weight $m$. Then there is a natural isomorphism
	$$\Hom_R(H_1(\u), V)\cong \Hom_R(\Gr^W_mH_1(\u), V)\cong \left\{
	\begin{array}{ll}
	0 & m\geq 0 \\
	H^1(\G, V) & m< 0 \\
	\end{array} 
	\right.
	.$$
	
\end{proposition}

\subsection{Application to the universal hyperelliptic curves}
Let $k$ be a field of characteristic zero such that the image of the $\ell$-adic cyclotomic character $\chi_\ell:G_k\to \Zl^\times$ is infinite.  Fix an algebraic closure $\bar k$ of $k$. Assume $g\geq 2$. Let $\bar\eta:\Spec\Omega\to\cH_{g,n+1/\bar k}$ be a geometric point of the stack $\cH_{g,n+1/\bar k}$. Denote by $\etabar_0$ the image of $\etabar$ in $\cH_{g/\bar k}$ and by $\bar{x}_j$ the image of $\etabar$ in $\cH_{g,1/\bar k}$ under the $j$th projection $\cH_{g,n+1/\bar k}\to \cH_{g,1/\bar k}$. 
Let $C$ be the fiber of the universal curve over $\bar\eta_0$. The fiber of $\cC^{n+1}_{\cH_g/k}\to\cH_{g/k}$ over $\etabar_0$ is the product $C^{n+1}$ with the base point $x_\etabar=(\bar{x}_0,\ldots,\bar{x}_n)$.  The fiber of $\cH_{g,n+1/k}\to \cH_{g/k}$ is $C^{n+1}-\Delta$, where $\Delta$ is the union of all diagonal divisors of $C^{n+1}$., i.e., $\Delta=\cup\Delta_{ij}$, where each $\Delta_{ij}$ consists of points of $C^{n+1}$ with $i$th and $j$th components equal.  The fiber $C^{n+1}-\Delta$ has also the base point $x_\etabar$. 
Denote by $\D^\geom_{g,n}$ the relative completion of $\pi_1(\cH_{g,n/\bar k},\bar\eta)$ with respect to the natural monodromy representation
$$\rho^\geom_\etabar:\pi_1(\cH_{g,n/\bar k},\bar\eta)\to\pi_1(\cH_{g/\bar k},\etabar_0)\to \Sp(H_\Ql),$$ 
where $H_\Ql=H^1_\et(C, \Ql(1))$. For the definition and properties of relative completion, see \cite{hain0, hain2, hain3}. 
Denote the prounipotent radical of $\D^\geom_{g,n}$ by $\cV^\geom_{g,n}$ and its Lie algebra by $\mathfrak{v}^\geom_{g,n}$. 
Denote the $\ell$-adic unipotent completions of $\pi_1(C, \bar{x}_j)$ and $\pi_{g,n}:=\pi_1(C^n-\Delta, \etabar)$ by  $\cP_j$ and $\cP_{g,n}$ and their Lie algebras by $\p_j$ and $\p_{g,n}$, respectively.  As convention, denote $\cP_0$ and $\p_0$ by $\cP$ and $\p$, respectively. \\
\indent  Recall that in section \ref{hyperelliptic J homo and Dehn twists}, we defined $\p$ to be the Lie algebra of the Malcev completion of the topological fundamental group $\Pi$. The Lie algebra of the $\ell$-adic unipotent completion here is obtained by base change to $\Ql$, and so by abuse of notation we denote it by $\p$ as well. 

\begin{variant}The universal curve over $\cH_{g,(1)/\bar k}$ induces a homomorphism 
$$\Psi: \cH_{g,(1)/\bar k}\to \cH_{g,1/\bar k}$$
that makes the diagram 
$$\xymatrix{
\cH_{g,(1)/\bar k}\ar[dr]\ar[d]_\Psi&\\
\cH_{g,1/\bar k}\ar[r]&\cH_{g/\bar k}
}
$$
commute.
Let $\bar\eta:\Spec\Omega\to\cH_{g,(1)/\bar k}$ be a geometric point of $\cH_{g,(1)/\bar k}$. We regard $\bar\eta$ as geometric points of $\cH_{g,1/\bar k}$ and $\cH_{g/\bar k}$ as well. Denote by $\D^\geom_{g,(1)}$ the relative completion of $\pi_1(\cH_{g,(1)/\bar k}, \bar\eta)$ with respect to the natural monodromy representation
$$\rho^w_\etabar:\pi_1(\cH_{g,(1)/\bar k}, \bar\eta)\to \pi_1(\cH_{g/\bar k}, \bar\eta)\to \Sp(H_\Ql).$$
That $\Delta_{g,(1)}$ is a finite-index subgroup of $\Delta_g$ implies that $\rho^w_\etabar$ has a Zariski-dense image. 
Denote the prounipotent radical of $\D^\geom_{g,(1)}$ by $\cV^\geom_{g,(1)}$ and its Lie algebra by $\mathfrak{v}^\geom_{g,(1)}$. We have the following diagram 
$$\xymatrix{
&&\mathfrak{v}^\geom_{g,(1)}\ar[dr]\ar[d]&&\\
0\ar[r]&\p\ar[r]&\mathfrak{v}^\geom_{g,1}\ar[r]&\mathfrak{v}^\geom_g\ar[r]&0,
}
$$
where the injectivity $\p\to \mathfrak{v}_{g,1}^\geom$ follows from the fact that the center of $\p$ is free (see \cite{M.A}).
\end{variant}

 Denote by $\D^\cC_{g,n}$, $\D_{g,n}$ and $\widehat{\D}_{g,n}$, respectively, the weighted completions of the fundamental groups $\pi_1(\cC_{\cH_{g,n/k}}, \etabar)$, $\pi_1(\cH_{g,n/k}, \bar\eta)$ and $\pi_1(\cC^n_{\cH_{g/k}},\bar\eta)$ with respect to 
 their natural  monodromy representations to $R:=\GSp(H_\Ql)$
and the central cocharacter $\omega:\Gm\to R$ defined by 
$a\mapsto a^{-1}\id.$ 
Denote the pronilpotent radicals of $\D^\cC_{g,n}$, $\D_{g,n}$ and $\widehat{\D}_{g,n}$ by $\cV^\cC_{g,n}$, $\cV_{g,n}$ and $\widehat{\cV}_{g,n}$, respectively. Denote the Lie algebras of $\D^\cC_{g,n}$, $\D_{g,n}$, $\widehat{\D}_{g,n}$, $\cV^\cC_{g,n}$, $\cV_{g,n}$, and $\widehat{\cV}_{g,n}$ by $\d^\cC_{g,n}$, $\d_{g,n}$, $\widehat{\d}_{g,n}$,  $\mathfrak{v}^\cC_{g,n}$, $\mathfrak{v}_{g,n}$, and $\widehat{\mathfrak{v}}_{g,n}$, respectively. The morphisms $\cH_{g,n/k}\to \cC^n_{\cH_g/K}\to (\cH_{g,1/k})^n$ induce a commutative diagram
$$\xymatrix{
	1\ar[r]& \pi_1(C^n-\Delta, x_\etabar)\ar[d]\ar[r]&\pi_1(\cH_{g,n/k}, \etabar)\ar[r]\ar[d]&\pi_1(\cH_{g/k}, \etabar_0)\ar[r]\ar@{=}[d]&1\\
	1\ar[r]&\prod_{j=1}^n\pi_1(C,\bar{x}_j)\ar[r]\ar@{=}[d]&\pi_1(\cC^n_{\cH_g/k}, \etabar)\ar[d]\ar[r]&\pi_1(\cH_{g/k}, \etabar_0)\ar[r]\ar[d]&1\\
	1\ar[r]&\prod_{j=1}^n\pi_1(C,\bar{x}_j)\ar[r]&\Pi_{j=1}^n\pi_1(\cH_{g,1/k},\bar{x}_j)\ar[r]&\pi_1(\cH_{g/k}, \etabar_0)^n\ar[r]&1.
}
$$
Denote the Lie algebra of the weighted completion of $\pi_1(\cH_{g,1/k},\bar{x}_j)$ by $\d^{(j)}_{g,1}$. The following proposition is the direct modification of \cite[Prop.~8.6]{hain2} to our case.
\begin{proposition}\label{comm diag}
	If $g\geq 2$, then the above diagram induces the commutative diagram
	$$\xymatrix{
		0\ar[r]& \p_{g,n}\ar[d]\ar[r]&\d_{g,n}\ar[r]\ar[d]&\d_g\ar[r]\ar@{=}[d]&1\\
		1\ar[r]&\bigoplus_{j=1}^n\p_j\ar[r]\ar@{=}[d]&\widehat{\d}_{g,n}\ar[d]\ar[r]&\d_g\ar[r]\ar[d]&1\\
		1\ar[r]&\bigoplus_{j=1}^n\p_j\ar[r]&\bigoplus_{j=1}^n\d^{(j)}_{g,1}\ar[r]&(\d_{g})^n\ar[r]&1.
	}
	$$ The same result holds for the Lie algebras of the relative completions of the corresponding geometric fundamental groups. 
	\qed
\end{proposition}

\begin{remark}
	Since the functor $\Gr^W_\bullet$ is exact, after applying $\Gr^W_\bullet$ to this diagram, the rows become exact sequences of $S_n\times \GSp(H_\Ql)$-modules. 
\end{remark}

By Theorem \ref{weighted property 1}, as a $\D_{g,n}$-module, the Lie algebra $\mathfrak{v}_{g,n}$ admits a natural weight filtration $W_\bullet\mathfrak{v}_{g,n}$ that satisfies the properties: $W_{-1}\mathfrak{v}_{g,n}=\mathfrak{v}_{g,n}$,  and the action of $\D_{g,n}$ on  $\Gr^W_{-m}\mathfrak{v}_{g,n}:=W_{-m}\mathfrak{v}_{g,n}/W_{-m-1}\mathfrak{v}_{g,n}$ factors through the action of $R$. Similarly for $\widehat{\d}_{g,n}$. The Lie algebras $\p$ and $\p_{g,n}$ are also negatively weighted as $\D_{g,n}$-modules, i.e., $W_{-1}\p=\p$ and $W_{-1}\p_{g,n}=\p_{g,n}$ and satisfy that $\Gr^W_{-1}H_1(\p)=H_1(\p)$ and $\Gr^W_{-1}H_1(\p_{g,n})=H_1(\p_{g,n})$ (see \cite{hain0}).
 When $n=1$, there is a commutative diagram of pronilpotent Lie algebras
$$\xymatrix{
&&\mathfrak{v}^\geom_{g,(1)}\ar[dr]\ar[d]&&\\
0\ar[r]&\p\ar@{=}[d]\ar[r]&\mathfrak{v}^\geom_{g,1}\ar[d]\ar[r]&\mathfrak{v}^\geom_g\ar[d]\ar[r]&0\\
0\ar[r]&\p\ar[r]&\mathfrak{v}_{g,1}\ar[r]&\mathfrak{v}_g\ar[r]&0.
}
$$
The adjoint action of $\mathfrak{v}^\geom_{g,(1)}$ on $\p$ factors as
$$\mathfrak{v}^\geom_{g,(1)}\to\mathfrak{v}^\geom_{g,1}\to\mathfrak{v}_{g,1}\to \Der\p.$$
\begin{variant}
In this paper, we will use the weighted completion of the fundamental group $\pi_1(\cH_{g,n/k}[r],\etabar)$.
For $r\geq 3$, it is represented by a smooth quasi-projective variety.  There is a unique conjugacy class of isomorphisms 
$$\pi_1(\cH_{g,n/\bar k}[r], \etabar)\cong (\Delta_{g,n}[r])^\wedge.$$ Denote the weighted completion of $\pi_1(\cH_{g,n/k}[r],\etabar)$ by $\D_{g,n}[r]$ and its prounipotent radical by $\cV_{g,n}[r]$. Denote their Lie algebras by $\d_{g,n}[r]$ and $\mathfrak{v}_{g,n}[r]$, respectively.  Proposition \ref{comm diag} holds for an abelain level $r\geq 1$ as well. 
\end{variant}
Unless stated otherwise, we set $R=\GSp(H_\Ql)$. An irreducible $R$-representation $V$ is said to be geometrically nontrivial when the restriction to $\Sp(H_\Ql)$ is nontrivial. 
\begin{proposition}\label{even weight}Suppose that $g\geq2$ and $V$ is a finite-dimensional irreducible $R$-representation of weight $m$. If
\begin{enumerate}
\item $r\geq 1$ and $V$ is geometrically nontrivial, then
$$\Hom_{R}(H_1(\mathfrak{v}_g[r]), V)\cong 
 \left\{
        \begin{array}{ll}
            H^1(\pi_1(\cH_{g/k}[r],\bar\eta), V) & m=-2\\
            0 & otherwise
        \end{array}
    \right.
$$
\item $V=\Ql(s)$, then
$$\Hom_{R}(H_1(\mathfrak{v}_g), \Ql(s))\cong 
 \left\{
        \begin{array}{ll}
        H^1(G_k, \Ql(s)) & s\geq 1\\
            0 & otherwise
        \end{array}
    \right.$$
\end{enumerate}
\end{proposition}
\begin{proof}Let $\V$ be  the  local system over $\cH^{c}_{g/k}[r]$  corresponding to the representation $V$. Associated to the inclusion $\cH_{g/k}[r]\to \cH^{c}_{g/k}[r]$, there is a Gysin sequence:
$$ \cdots \to H^1_\et(\cH^{c}_{g/\bar k}[r], \V)\to H^1_\et(\cH_{g/\bar k}[r], \V)\to \bigoplus _iH^0_\et(D_i, \V(-1))\to \cdots,$$ where $\bigcup_iD_i=\cH^{c}_{g/\bar k}[r]-\cH_{g/\bar k}[r]$. The representation $V$ is defined over $\Q$ and denote by $V_\Q$  the restriction of $V$ to $\Q$. We have $V=V_\Q\otimes \Ql$. 
The simply-connectedness of $\cH^c_g[0]$ implies that there is an isomorphism
$H^1(\cH^{c}_g[r](\C), V_\Q)\cong H^1(G_g[r], V_\Q)$, which is trivial by a theorem of Borel in \cite{bor1,bor2}.  Since there is an isomorphism $H^1(\cH^{c}_g[r], V_\Q)\otimes \Ql\cong H^1_\et(\cH^{c}_{g/\bar k}[r], \V)$, it follows that $H^1_\et(\cH_{g/\bar k}[r], \V)$ is a $G_k$-representation of weight $2+m$. Since $\cH_{g/ k}[r]$ is a finite Galois cover of $\cH_{g/k}$, it then also follows that $H^1(\pi_1(\cH_{g/\bar k}, \bar\eta), V)$ is of weight $2+m$. Recall that there is an exact sequence
$$1\to\pi_1(\cH_{g/\bar k}[r], \bar\eta)\to \pi_1(\cH_{g/k}[r], \bar\eta)\to G_k\to1.$$
Denote the groups $\pi_1(\cH_{g/\bar k}[r], \bar\eta)$ and $\pi_1(\cH_{g/k}[r], \bar\eta)$ by $\Delta^\geom[r]$ and $\Delta^\arith[r]$, respectively. When $r=1$, we omit $r$ from the notation.  The above exact sequence gives rise to a spectral sequence
$$E^{s,t}_2=H^s(G_k, H^t(\Delta^\geom[r], V))\Rightarrow H^{s+t}(\Delta^\arith[r], V).$$
Thus there is an exact sequence
$$0\to H^1(G_k,H^0(\Delta^\geom[r],V))\to H^1(\Delta^\arith[r], V)\hspace{2in}$$
$$\hspace{1.5in}\to H^0(G_k, H^1(\Delta^\geom[r], V))\to H^2(G_k, H^0(\Delta^\geom[r], V)).$$ 
If $V$ is geometrically nontrivial, then since $\rho^\geom:\Delta^\geom[r]\to \Sp(H_\Ql)$ has a Zariski-dense image, we have $H^0(\Delta^\geom[r],V)=0$. Hence we have an isomorphism\\ $H^1(\Delta^\arith[r], V)\cong H^1(\Delta^\geom[r], V)^{G_k}$.  Thus if $|V|\not=2$, then the result follows. When $|V|=-2$, the result follows from Proposition \ref{generator iso}. \\
For $V=\Ql(s)$, since $H_1(\Delta_g,\Z)$ is torsion by \cite[Thm.~8]{BiHi}, it follows that 
$$H^1(\Delta^\geom, \Ql(s))\cong H^1(\Delta_g, \Q)\otimes\Ql(s)=0.$$ Therefore, the above exact sequence gives the isomorphism 
$$H^1(G_k, \Ql(s))\cong H^1(\Delta^\arith, \Ql(s)).$$ Thus if $s\geq 1$, then the result follows. 
\end{proof}
\begin{proposition}\label{odd weight is zero}
For $g\geq 2$, each graded quotient $\Gr^W_{-2m-1}\mathfrak{v}_g=0$ for $m\geq0$.
\end{proposition}
\begin{proof} 
Proposition \ref{even weight} implies that $H_1(\mathfrak{v}_g)$ has only negative even weights. Since a pronilpotent Lie algebra $\n$ is generated by $H_1(\n)$, we are done. 
\end{proof}
The following result is an immediate consequence of Tanaka's computation \cite{tanaka}.
\begin{lemma}\label{no nichi}
If $g\geq 2$, then $\Hom_{R}(\Gr^W_{-2}\mathfrak{v}_g, V_{[1^2]})=0$.
\end{lemma}
\begin{proof}We use the same notation as in the proof of Proposition \ref{even weight}. Since $\mathfrak{v}_g=W_{-2}\mathfrak{v}_g$, there is an $R$-invariant isomorphism 
$$\Gr^W_{-2}\mathfrak{v}_g=\Gr^W_{-2}H_1(\mathfrak{v}_g),$$
since $\mathfrak{v}_g$ is generated by $H_1$ and $\Gr^W_{-1}H_1(\mathfrak{v}_g)=0$. Tanaka's computation \cite{tanaka} implies that $H^1(\Delta^\geom, V_{[1^2]})$ vanishes. Therefore, together with the isomorphism
$$\Hom_{R}(\Gr^W_{-2}H_1(\mathfrak{v}_g), V_{[1^2]})\cong H^1(\Delta^\arith, V_{[1^2]}),$$ the result follows.
\end{proof}
\begin{lemma}For $g\geq 2$, $\Gr^W_{-2}\mathfrak{v}_g$ contains at least one copy of $V_{[2^2]}(-1)$.
\end{lemma}
\begin{proof}

There is a diagram of $\D_{g,1}$-modules: $$\xymatrix{
0\ar[r]&\p\ar[d]\ar[r]&\mathfrak{v}_{g,1}\ar[d]\ar[r]&\mathfrak{v}_g\ar[d]\ar[r]&0\\
0\ar[r]&\Inn\p\ar[r]& W_{-1}\Der\p\ar[r]&W_{-1}\Out\Der\p\ar[r]&0
}
$$
The Lie algebra $\p$ admits a natural weight filtration and so does $\Der\p$. That $H_1(\p) $ is pure of weight $-1$ and the strictness of morphisms in the category of $\D_{g,1}$-modules imply  that this natural weight filtration coincides with its lower central series. Thus we have $\Gr^W_{-k}\p=\p(k)$ and $\Gr^W_{-k}\Der\p=\Der_k\p$.  
By the construction of relative and weighted completions, the hyperelliptic Johnson homomorphism $\tau^\hyp: T\Delta_g\to \Der_{2}\p$ factors through the $R$-invariant map $\Gr^W_{-2}H_1(\mathfrak{v}_{g,1})\to \Der_{2}\p$ that is induced by the adjoint action of $\mathfrak{v}_{g,1}$ on $\p$.  The proof of Theorem \ref{no zero outer part} shows that the composition $T\Delta_g\to H_1(\mathfrak{v}_{g,1})\to \Der_{2}\p\to\Out\Der_{2}\p$ is nontrivial.  Therefore the map $H_1(\mathfrak{v}_g)\to \Out\Der_2\p=V_{[2^2]}(-1)$ is nontrivial and hence the $R$-module $\Gr^W_{-2}\mathfrak{v}_g=\Gr^W_{-2}H_1(\mathfrak{v}_g)$ contains at least one copy of $V_{[2^2]}(-1)$ by Schur's lemma.
\end{proof}

\section{The Lie algebra $\mathfrak{h}_{g,n}$ }

 From Proposition \ref{comm diag}, we see that the projection $\cC_{\cH_{g/k}}^n\to \cH_{g/k}$ induces 
the exact sequence of $\widehat{\D}_{g,n}$-modules
$$0\to \p^n\to \widehat{\mathfrak{v}}_{g,n}\to\mathfrak{v}_g\to 0.$$
 Since $\Gr^W_\bullet$ is an exact functor, we have
\begin{proposition}
For $g\geq 2$ and $n\geq 0$, we have exact sequences of $R$-modules
\begin{align*}
0\to &(\Gr^W_{-1}\p)^n\to  \Gr^W_{-1}\widehat{\mathfrak{v}}_{g,n}\to 0;\\
0\to &(\Gr^W_{-2}\p)^n\to \Gr^W_{-2}\widehat{\mathfrak{v}}_{g,n}\to\Gr^W_{-2}\mathfrak{v}_g\to 0;\\ 
0\to &(\Gr^W_{-3}\p)^n\to \Gr^W_{-3}\widehat{\mathfrak{v}}_{g,n}\to 0;\\
0\to&(\Gr^W_{-4}\p)^n\to \Gr^W_{-4}\widehat{\mathfrak{v}}_{g,n}\to \Gr^W_{-4}\mathfrak{v}_g\to 0.
\end{align*}
\qed
\end{proposition}
Assume that $g\geq 2$ and $n\geq 0$. We define the Lie algebra $\h_{g,n}$ to be 
$$\h_{g,n}=\Gr^W_\bullet\left(\widehat{\mathfrak{v}}_{g,n}/W_{-5}\widehat{\mathfrak{v}}_{g,n}\right).$$
Denote the weight $k$ component of $\h_{g,n}$ by $\h_{g,n}(k)$. We have $$\h_{g,n}(k)=\Gr^W_{k}\left(\widehat{\mathfrak{v}}_{g,n}/W_{-5}\widehat{\mathfrak{v}}_{g,n}\right).$$
\begin{proposition}\label{induced sections}Suppose that $g\geq 2$. Each section of the universal hyperelliptic curve $\cC_{\cH_{g,n/k}}\to \cH_{g,n/k}$ induces a well-defined $R$-invariant Lie algebra section of $\beta_n:\h_{g,n+1}\to\h_{g,n}.$
\end{proposition}
\begin{proof}
The fiber product diagram
$$\xymatrix{
\cC_{\cH_{g,n}/k}\ar[r]\ar[d]&\cC_{\cH_g/k}^{n+1}\ar[d]\\
\cH_{g,n/k}\ar[r]          &\cC_{\cH_g/k}^n
}
$$ induces the following commutative diagram of $\D^\cC_{g,n}$-modules
$$\xymatrix{
0\ar[r]&\p\ar[r]\ar@{=}[d]&\widehat{\mathfrak{v}}_{g,n+1}\ar[r]&\widehat{\mathfrak{v}}_{g,n}\ar[r]&0\\
0\ar[r]&\p\ar[r]&\mathfrak{v}^\cC_{g,n}\ar[r]\ar[u]&\mathfrak{v}_{g,n}\ar[u]\ar[r]&0,
}
$$ where the rows are exact and the middle and right vertical maps are surjective. By the universal property of weighted completion, each section $s$ of $\cC_{\cH_{g,n}/k}\to \cH_{g,n/k}$ induces an $R$-invariant graded Lie algebra section $\Gr^W_\bullet ds_{\ast}$ of $\Gr^W_\bullet\mathfrak{v}^\cC_{g,n}\to \Gr^W_\bullet\mathfrak{v}_{g,n}$. Hence it induces a section $\Gr^W_\bullet ds_{\ast}/W_{-5}$ of
$ \Gr^W_\bullet(\mathfrak{v}^\cC_{g,n}/W_{-5})\to \Gr^W_\bullet(\mathfrak{v}_{g,n}/W_{-5}).$
 Denote the kernel of $\mathfrak{v}_{g,n}\to \widehat{\mathfrak{v}}_{g,n}$ by $\mathfrak{k}$.
 It follows from the commutative diagram in Proposition \ref{comm diag} that $\mathfrak{k}$ is equal to the kernel of canonical surjection $\p_{g,n}\to \p^n:=\bigoplus_{j=1}^n\p_j$. We observe that the restriction of $\Gr^W_\bullet ds_{\ast}$ to $\Gr^W_\bullet \mathfrak{k}$ maps into $\Gr^W_\bullet\p$. We claim that this map is trivial. It follows from \cite[Thm.~12.6]{hain0} that $\Gr^W_\bullet\mathfrak{k}$ is generated as a Lie ideal in $\Gr^W_\bullet\p_{g,n}$ by $\ker(\Gr^W_{-2}\p_{g,n}\to \Gr^W_{-2}\p^n)=\Gr^W_{-2}\mathfrak{k}=\bigoplus_{1\leq i<j\leq n}\Ql(1)$. By Proposition \ref{presentation}, the restriction $\Gr^W_{-2}ds_{\ast}$ to $\Gr^W_{-2}\mathfrak{k}$ is trivial, and hence the restriction of $\Gr^W_\bullet ds_{\ast}$ to $\Gr^W_\bullet\mathfrak{k}$ is trivial. 
 Therefore, the Lie algebra section $\Gr^W_\bullet ds_{\ast}$ of  $\Gr^W_\bullet\mathfrak{v}^\cC_{g,n}\to \Gr^W_\bullet\mathfrak{v}_{g,n}$ descends to a Lie algebra section of $\Gr^W_\bullet\widehat{\mathfrak{v}}_{g,n+1}\to \Gr^W_\bullet\widehat{\mathfrak{v}}_{g,n}$. In particular, it induces a Lie algebra section of $\beta_n:\h_{g,n+1}\to\h_{g,n}$.

\end{proof}

Since each $\Gr^W_{-m}\widehat{\mathfrak{v}}_{g,n}$ is the inverse limit of finite-dimensional $R$-modules, we may consider $\Gr^W_{-m}\mathfrak{v}_g$ as an $R$-submodule of $\Gr^W_{-m}\widehat{\mathfrak{v}}_{g,n}$ by choosing a continuous $R$-module section. More precisely, we fix a splitting as follows. The Lie algebra $\p^n$ injects into $W_{-1}\Der\p^n$ by the adjoint action, since the center of $\cP$ is free. For each $r\in\Gr^W_{-m}\widehat{\mathfrak{v}}_{g,n}$, we have $\ad(r)=\mathrm{inn}(r)+\mathrm{out}(r)\in \Inn\Der_m\p^n\oplus \Out\Der_m\p^n$. Then there exists a unique vector $r_\p\in \Gr^W_{-m}\p^n$ such that $\ad(r_\p)=\mathrm{inn}(r)$. This gives a splitting $\Gr^W_{-m}\widehat{\mathfrak{v}}_{g,n}\to \Gr^W_{-m}\p^n$ taking $r\mapsto r_\p$. In this way, for each $m\geq 1$, we fix a splitting
$$\Gr^W_{-m}\widehat{\mathfrak{v}}_{g,n}= \Gr^W_{-m}\p^n\oplus \Gr^W_{-m}\mathfrak{v}_g.$$

Applying Theorem \ref{no zero outer part}, we show that the bracket of $\Gr^W_{-2}\mathfrak{v}_g$ with itself in $\Gr^W_\bullet\widehat{\mathfrak{v}}_{g,n}$ maps nontrivially into $\Gr^W_{-4}\p^n$. This nontrivial bracket gives an obstruction for an $R$-invariant section of each graded quotient to be an $R$-invariant Lie algebra section. 
\begin{proposition}\label{nonzero inner bracket} If $g\geq 3$ and $n\geq 1$, then there are nonzero elements $\delta_n, \tilde{\delta}_n\in\Gr^W_{-2}\mathfrak{v}_g\subset \Gr^W_{-2}\widehat{\mathfrak{v}}_{g,n}$ such that the bracket $[\delta_n, \tilde{\delta}_n]$ is nonzero and  maps diagonally into $(\Gr^W_{-4}\p)^n\subset\Gr^W_{-4}\widehat{\mathfrak{v}}_{g,n}$.
\end{proposition}
\begin{proof}
Recall notation from \ref{The outer action of a commuting pair of Dehn twists}. Notice that the hyperelliptic Torelli group $T\Delta_g$ maps to $\cV_{g,(1)}^\geom$. Composing with the logarithm map $
\log:\cV_{g,(1)}^\geom\to\mathfrak{v}^\geom_{g,(1)}$, $T\Delta_g$ maps to the pronilpotent Lie algebra $\mathfrak{v}_{g,(1)}^\geom$. The adjoint action of $\mathfrak{v}_{g,(1)}^\geom$ on $\p$ factors through the adjoint action of $\mathfrak{v}_{g,1}$ on $\p$. Denote the image of the commuting Dehn twists $\omega_1$ and $\omega_{g-1}$ in $\mathfrak{v}_{g,1}$ by $\omega$ and $\tilde{\omega}$, respectively. The fact that the Johnson homomorphism $T\Delta_g\to \Der_1\p$ is trivial implies that $\Gr^W_{-1}\omega=\Gr^W_{-1}\tilde{\omega}=0$. Since $\Gr^W_{-2}\mathfrak{v}_{g,1}=\Gr^W_{-2}\p\oplus\Gr^W_{-2}\mathfrak{v}_g$, we can express 
$$\Gr^W_{-2}\omega=\delta_{\p}+\delta_{\mathfrak{v}},\text{ and }\Gr^W_{-2}\tilde{\omega}=\tilde{\delta}_{\p}+\tilde{\delta}_{\mathfrak{v}}.$$ Since $\Gr^W_{-2}\mathfrak{v}_g$ does not contain a copy of $V_{[1^2]}$ by Lemma \ref{no nichi}, we can see that the elements $\delta_{\mathfrak{v}}$ and $\tilde{\delta}_{\mathfrak{v}}$ map to $\xi_{[2^2]}$ and $\tilde{\xi}_{[2^2]}$, respectively (recall that $\xi_{[2^2]}$ and $\tilde{\xi}_{[2^2]}$ are defined in \ref{The outer action of a commuting pair of Dehn twists}). Since the Dehn twists $\omega$ and $\omega_{g-1}$ commutes, the bracket $[\omega, \tilde{\omega}]$ in the Lie algebra $\mathfrak{v}_{g,1}$ is trivial. Since  $\Gr^W_{-1}\omega=\Gr^W_{-1}\tilde{\omega}=0$, we have $$\Gr^W_{-4}[\omega, \tilde{\omega}]=[\delta_{\p}+\delta_{\mathfrak{v}},\tilde{\delta}_{\p}+\tilde{\delta}_{\mathfrak{v}}]=0.$$
Therefore, the bracket $[\delta_{\mathfrak{v}}, \tilde{\delta}_{\mathfrak{v}}]$ lies in $\Gr^W_{-4}\p$. By Theorem \ref{no zero outer part}, the bracket $[\xi_{[2^2]},\tilde{\xi}_{[2^2]}]$ is nontrivial, and so is the bracket $[\delta_{\mathfrak{v}}, \tilde{\delta}_{\mathfrak{v}}]$ in $\Gr^W_{-4}\p$. This completes the case for $n=1$. For $n>1$, the diagonal section $\cC_{\cH,g}^1\to \cC_{\cH,g/k}^n$ induces a graded section $\Gr^W_\bullet d\Delta$ of the projection $\Gr^W_\bullet \widehat{\mathfrak{v}}_{g,n}\to \Gr^W_\bullet\widehat{\mathfrak{v}}_{g,1}$. Note that the restriction of $\Gr^W_{-2} d\Delta$ to $\Gr^W_{-2}\mathfrak{v}_g$ is the identity. Let $\delta_n=(\Gr^W_{-2}d\Delta)(\delta_{\mathfrak{v}})$ and $\tilde{\delta}_n=(\Gr^W_{-2}d\Delta)(\tilde{\delta}_{\mathfrak{v}})$. Then the bracket $[\delta_n,\tilde{\delta}_n]$ maps to $(\Gr^W_{-4}\p)^n$ diagonally, since $\Gr^W_\bullet d\Delta$ is a graded Lie algebra section. 
The bracket $[\delta_n,\tilde{\delta}_n]$ maps to $[\delta_{\mathfrak{v}}, \tilde{\delta}_{\mathfrak{v}}]$, which is nonzero, and so is $[\delta_n,\tilde{\delta}_n]$. 
\end{proof}
\subsection{The sections of $\beta_n:\h_{g,n+1}\to\h_{g,n}$}
In this section, we will compute $R$-invariant Lie algebra sections of $\beta_n$ that are induced by the sections of the universal hyperelliptic curve $\cC_{\cH_{g,n}/k}\to \cH_{g,n/k}$. We call them {\it geometric} sections of $\beta_n$.
 Remember that $\p^{n+1}=\oplus_{j=0}^n\p_j\subset\widehat{\mathfrak{v}}_{g,n+1}$, where the $j$th component $\p_j$ corresponds to the projection of $\cC^{n+1}_{\cH_{g/k}}$ onto its $j$th component. We observe that, by Schur's lemma, each section $\zeta$ of $\beta_n$ on $\Gr^W_{-1}$ has the form
$$(\Gr^W_{-1}\zeta)(u_1,\ldots, u_n)=(\sum_{j=1}^na_ju_j; u_1,\ldots,u_n)\in \Gr^W_{-1}\p_0\oplus\bigoplus_{j=1}^n\Gr^W_{-1}\p_j$$
where the $a_j$ are some constants in $\Ql$. From the graded Lie algebra structure of $\h_{g,n}$, we determine all the $R$-invariant Lie algebra sections of $\beta_n$.

\begin{theorem}\label{sections of lie algebras} Suppose $g\geq 3$. There are exactly $2n$ $R$-invariant Lie algebra sections of $\beta_n$, $\zeta_1^{\pm},\ldots,\zeta_n^{\pm}$, where $\zeta_j^{\pm}$ on $\Gr^W_{-1}$ is given by $$\Gr^W_{-1}\zeta_j^{\pm}:(u_1,\ldots,u_n)\to (\pm u_j;u_1,\ldots,u_n).$$ In particular, $\beta_0$ has no $R$-invariant Lie algebra section. 

\end{theorem}
\begin{proof}
	Let $\zeta$ be an $R$-invariant Lie algebra section of $\beta_n$. 
Since $\zeta$ is a Lie algebra section and the bracket $[~,~];\Lambda^2\Gr^W_{-1}\p\to\Gr^W_{-2}\p$ is surjective, we see that $\Gr^W_{-2}\zeta$ is given by
$$\Gr^W_{-2}\zeta:(v_1,\ldots,v_n; r)\mapsto (\sum_{j=1}^na^2_jv_j;v_1,\ldots,v_n;r)\in\Gr^W_{-2}\p_0\oplus\bigoplus_{j=1}^n\Gr^W_{-2}\p_j\oplus\Gr^W_{-2}\mathfrak{v}_g,$$
where  each $v_j$ is in $\Gr^W_{-2}\p_j$ and $r$ in $\Gr^W_{-2}\mathfrak{v}_g$. Note that there is no $R$-map from the component $\Gr^W_{-2}\mathfrak{v}_g$ to $\Gr^W_{-2}\p_0\cong V_{[1^2]}$ by Lemma \ref{no nichi}.
Similarly we see that $\Gr^W_{-3}\zeta$ is given by
$$\Gr^W_{-3}\zeta: (w_1,\ldots,w_n)\mapsto (\sum_{j=1}^na^3_jw_j;w_1,\ldots,w_n)\in\Gr^W_{-3}\p_0\oplus\bigoplus_{j=1}^n\Gr^W_{-3}\p_j,$$
where each $w_j$ is in $\Gr^W_{-3}\p_j$. 

\begin{lemma}With notation as above.  Then we have $a_j=0,1$, or $-1$ for each $j=1,\ldots,n$.
\end{lemma}
\begin{proof} Here we omit Tate twists.  Recall that we have $\Der_{1}\p=V_{[1^3]}+V_{[1]}$ and $\Der_{2}\p=V_{[2^2]}+V_{[1^2]}$. The bracket $[~,~]:V_{[1]}\otimes V_{[2^2]}\to V_{[2+1]}\subset\Der_{3}\p$ is not trivial. We can see this, for example,  by evaluating the bracket 
	$$\left[\phi'(a_2\wedge\theta), \phi\left((a_1\wedge b_1)^2-\frac{3}{g+1}(a_1\wedge b_1)\cdot \theta\right)\right]$$ 
	at $a_2$, where $\phi$ is defined in section \ref{image of a dehn twist} and $\phi':\Lambda^3H\to\Der_1\p$ is given by
	$$x\wedge y\wedge z\mapsto (u\mapsto\theta(x,u)[y,z]+\theta(y,u)[z,x]+\theta(z,u)[x,y]).$$
Recall from the proof of Theorem \ref{no zero outer part} that  $\phi\left((a_1\wedge b_1)^2-\frac{3}{g+1}(a_1\wedge b_1)\cdot \theta\right)$ is in the copy of $V_{[2^2]}$ in $\Der_2\p$. The copy of $V_{[1]}$ in $\Der_1\p$ can be identified with $\phi'(H\wedge\theta)$. We have 
$$\left[\phi'(a_2\wedge\theta), \phi\left((a_1\wedge b_1)^2-\frac{3}{g+1}(a_1\wedge b_1)\cdot \theta\right)\right](a_2)=\frac{3}{g+1}[[[a_1,b_1],a_2],a_2],$$
which can be mapped to the highest weight vector of the copy of $V_{[3+1]}$ in $\Gr^W_{-4}\p$, and hence it is nonzero. Therefore, the bracket $[~,~]:V_{[1]}\otimes V_{[2^2]}\to V_{[2+1]}\subset\Der_{3}\p$ is surjective. This implies that the bracket 
$$[~,~]:\Gr^W_{-1}\p_j\otimes \Gr^W_{-2}\mathfrak{v}_g\to \Gr^W_{-3}\p_j $$
is also surjective. Therefore, we also have
$$\Gr^W_{-3}\zeta: (w_1,\ldots,w_n)\mapsto (\sum_{j=1}^na_jw_j,w_1,\ldots,w_n)\in\Gr^W_{-3}\p_0\oplus\bigoplus_{j=1}^n\Gr^W_{-3}\p_j.$$ Thus we obtain the relations $a_j^3=a_j$ or $a_j(a_j^2-1)=0$ for each $j=1,\ldots,n$.
\end{proof}

By Proposition \ref{nonzero inner bracket}, there are elements 
$\delta_n, \tilde{\delta}_n\in \Gr^W_{-2}\mathfrak{v}_g\subset\h_{g,n}(-2)$
such that the bracket $[\delta_n,\tilde{\delta}_n]$ is nontrivial and maps diagonally to $\bigoplus_{j=1}^n\Gr^W_{-4}\p_j$. Denote by $x_j$ the $j$th component of $[\delta_n,\tilde{\delta}_n]$ in $\Gr^W_{-4}\p_j$. The $x_j$ are all equal under the identification $\p_j\cong \p$ and denote the common element by $x$. Let $V$ be the irreducible $R$-subrepresentation containing $x$ of $\Gr^W_{-4}\p$.  Then the restriction of $\Gr^W_{-4}\zeta$ to  $V^n\subset\bigoplus_{j=1}^n\Gr^W_{-4}\p_j\subset \h_{g,n}(-4)$ is given by
$$(\Gr^W_{-4}\zeta)(t_1,\ldots,t_n)=(\sum_{j=1}^na_j^4t_j,t_1,\ldots,t_n).$$ Since $\zeta(\delta_n)=\delta_{n+1}$ and $\zeta(\tilde{\delta}_n)=\tilde{\delta}_{n+1}$, we have
$$(\Gr^W_{-4}\zeta)(x_1,\ldots,x_n)=(x_0,x_1,\ldots,x_n),$$
where $x_j=[\delta_n,\tilde{\delta}_n]$ for each $j=0,\ldots,n$. Hence we have $x=\sum_{j=1}^na_j^4x$ and so the relation $\sum_{j=1}^na_j^4=1$. This implies that $a_j^2=1$ for some $j$ and $a_i=0$ for $i\not=j$. 
In $\Gr^W_{-4}$, the section $\Gr^W_{-4}\zeta$ is given by 
$$\Gr^W_{-4}\zeta:(t_1,\ldots,t_n; s)\to \left(\phi_{\zeta}(s)+t_j; t_1,\ldots,t_n;s\right)\in \Gr^W_{-4}\p_0\oplus\bigoplus_{j=1}^n\Gr^W_{-4}\p_j\oplus\Gr^W_{-4}\mathfrak{v}_g,$$
where $\phi_{\zeta}:\Gr^W_{-4}\mathfrak{v}_g\to\Gr^W_{-4}\p$ is the $R$-invariant map determined by $\zeta$. We claim that for any $\zeta$, the map $\phi_\zeta$ is trivial. There is an $R$-decomposition $\Gr^W_{-4}\mathfrak{v}_g=\Gr^W_{-4}\mathfrak{v}_g'\oplus \Gr^W_{-4}H_1(\mathfrak{v}_g),$ where $\mathfrak{v}_g'=[\mathfrak{v}_g,\mathfrak{v}_g]$, by fixing a continuous $R$-invariant section of $\mathfrak{v}_g\to H_1(\mathfrak{v}_g)$.  Say $s=s'+s^\ab$, where $s'\in \Gr^W_{-4}\mathfrak{v}_g'$ and $s^\ab\in\Gr^W_{-4}H_1(\mathfrak{v}_g).$ The element $s'$ is a sum of the elements of the form $[v,u]$ for $v,u\in \Gr^W_{-2}\mathfrak{v}_g$. We may assume that $s'=[v,u]$. Viewing the bracket $[v,u]$ in $\h_{g,n}$, we have $[v,u]=t_1+\cdots +t_n+s',$ where $t_1=\cdots=t_n=t$ are elements in $\Gr^W_{-4}\p$ determined by the bracket on $\h_{g,n}$. Then since $\zeta$ is a Lie algebra section, it follows that  $\Gr^W_{-4}\zeta(t_1,\ldots,t_n;s')=(t_0,t_1,\ldots,t_n;s')$, which implies that $t= \phi_\zeta(s')+t$, or $\phi_\zeta(s')=0$. That $\phi_\zeta(s^\ab)=0$ follows from Proposition \ref{even weight}. Thus each section of $\beta_n$ is determined by its effect on $\Gr^W_{-1}$. This completes the proof. 
\end{proof}
\section{The sections of the universal hyperelliptic curves}
In this section, we prove that the sections of the universal hyperelliptic curves are exactly the tautological ones and their hyperelliptic conjugates. Let $H=H_\Ql$. Each section $x$ of $\cC_{\cH_{g,n}/k}\to \cH_{g,n/k}$ gives a characteristic class $\kappa_{x}$ in $H^1_\et(\cH_{g,n/k}, \H_\Ql)\cong H^1(\pi_1(\cH_{g,n/k}), H)$. This class $\kappa_x$ is the pullback of the universal class $\kappa$ associated to the tautological section of $\cC_{g,1/k}\to \M_{g,1/k}$. For the construction of the universal class $\kappa$, see \cite{hm1}. For each $j=1\ldots,n$, denote by $\kappa_j$ the characteristic class corresponding to the $j$th tautological section of $\cC_{\cH_{g,n}/k}\to \cH_{g,n/k}$. On the other hand, the section $x$ induces a section $\zeta_x$ of $\beta_n$ by Proposition \ref{induced sections} and it has to be one of the $\zeta^\pm_j$ by Theorem \ref{sections of lie algebras}. Thus $\zeta_x$ is uniquely determined by its action on $\Gr^W_{-1}$. Firstly, we see that this action is determined by the cohomology class $\kappa_x$. 
\begin{proposition}\label{weight -1} If $g\geq 2$, then there are $R$-invariant isomorphisms 
	$$\h_{g,n}(-1)\cong \bigoplus_{j=1}^n\Gr^W_{-1} \p_j\cong \Gr^W_{-1}H_1(\widehat{\mathfrak{v}}_{g,n})\cong \Gr^W_{-1}H_1(\mathfrak{v}_{g,n}).$$
\end{proposition}
\begin{proof}
	Recall from Proposition \ref{comm diag} that there is the exact sequence of $\widehat{\D}_{g,n}$-modules
	$$0\to \oplus_{j=1}^n\p_j\to \widehat{\mathfrak{v}}_{g,n}\to\mathfrak{v}_g\to 0, $$
	which induces the exact sequence $\oplus_{j=1}^nH_j\to H_1(\widehat{\mathfrak{v}}_{g,n})\to H_1(\mathfrak{v}_g)\to 0$. The natural map $\cC_{\cH_{g/k}}^n\to (\cH_{g,1/k})^n$ induces the $\widehat{\D}_{g,n}$-module map $H_1(\widehat{\mathfrak{v}}_{g,n})\to \oplus_{j=1}^nH_1(\mathfrak{v}_{g,1})_j$, where the $j$th component corresponds to the $j$th projection of $\cC_{\cH_{g/k}}^n\to \cH_{g,1/k}$.     Then there is a commutative diagram
$$\xymatrix{
	\oplus_{j=1}^n H_j\ar[r]\ar@{=}[d]&H_1(\widehat{\mathfrak{v}}_{g,n})\ar[r]\ar[d]&H_1(\mathfrak{v}_g)\ar[r]\ar[d]^{\text{diag}}&0\\
	\oplus_{j=1}^n H_j\ar[r]&\oplus_{j=1}^nH_1(\mathfrak{v}_{g,1})_j\ar[r]&\oplus H_1(\mathfrak{v}_g)^n\ar[r]&0,
}
$$ where the rows are exact. We claim that the sequence
$0\to H\to H_1(\mathfrak{v}_{g,1})\to H_1(\mathfrak{v}_g)\to 0$ induced from  the exact  sequence $ 0\to\p\to\mathfrak{v}_{g,1}\to\mathfrak{v}_g\to0$ is exact. The composition $\p\to \mathfrak{v}_{g,1}\to W_{-1}\Der\p$ obtained by the adjoint action of $\mathfrak{v}_{g,1}$ on $\p$ is injective. We observe that the induced map $H\to H_1(\mathfrak{v}_{g,1})\to H_1(W_{-1}\Der\p)\to \Der_{1}\p$ agrees with the map $\Gr^W_{-1}\p=H\to \Gr^W_{-1}\mathfrak{v}_{g,1}\to \Der_{1}\p$, which is injective by the exactness of $\Gr^W_\bullet$. Thus the injectivity $H\to H_1(\mathfrak{v}_{g,1})$ follows, from which the injectivity of $\oplus_{j=1}^nH_j\to H_1(\widehat{\mathfrak{v}}_{g,n})$ follows. Since the functor $\Gr^W_\bullet$ is exact and $\Gr^W_{-1}H_1(\mathfrak{v}_g)=0$, we have that $\Gr^W_{-1}H_1(\widehat{\mathfrak{v}}_{g,n})\cong \oplus_{j=1}^nH_j$. Now, the open immersion $\cH_{g,n/k}\to \cC^n_{\cH_g/k}$ induces the surjection $\mathfrak{v}_{g,n}\to \widehat{\mathfrak{v}}_{g,n}$, which is an isomorphism on $\Gr^W_{-1}$. 
\end{proof}
\begin{remark} The same statement holds for $\widehat{\mathfrak{v}}_{g,n}[r]$ and $\mathfrak{v}_{g,n}[r]$ with $r \geq 1$. 
\end{remark}
By Proposition \ref{generator iso}, there are isomorphisms 
$$\Hom_{R}(\Gr^W_{-1}H_1(\widehat{\mathfrak{v}}_{g,n}), H)\cong \Hom_{R}(\Gr^W_{-1}H_1(\mathfrak{v}_{g,n}), H)\cong H^1(\pi_1(\cH_{g,n/k}), H).$$
It follows from Proposition \ref{weight -1} that  
$H^1(\pi_1(\cH_{g,n/k}), H)\cong H^1_\et(\cH_{g,n/k}, \H_\Ql)\cong \oplus_{j=1}^n\Ql_j$. In fact, it follows from the analogous result for $\M_{g,n/k}$ \cite[Prop. 12.1, Cor. 12.6]{hain2} that, for $g\geq3$, there is an isomorphism $$H^1_\et(\cH_{g,n/k}, \H_\Ql)\cong \oplus_{j=1}^n\Ql\kappa_j.$$ Let $p_0$ be the projection onto the $0$th component of $\h_{g,n+1}(-1)$. The naturality of the isomorphism $\Hom_{R}(\Gr^W_{-1}H_1(\mathfrak{v}_{g,n}), H)\cong H^1(\pi_1(\cH_{g,n/k}), H)$ and \cite[Cor. 12.6]{hain2} give the following fact:

\begin{proposition} Suppose $g\geq 3$. Let $x$ be a section of $\cC_{\cH_{g,n}/k}\to \cH_{g,n/k}$.
	Then under the isomorphism $\Hom_{R}(\Gr^W_{-1}H_1(\widehat{\mathfrak{v}}_{g,n}), H)\cong H^1_\et(\cH_{g,n/k}, \H_\Ql)$  the composition $p_0\circ\Gr^W_{-1}\zeta_x$ corresponds to $\frac{1}{2g-2}\kappa_x$.  \qed
\end{proposition}
Then Theorem \ref{sections of lie algebras} and \cite[Cor. 12.6]{hain2} let us determine which classes in\\ $H^1_\et(\cH_{g,n/k}, \H_\Ql)$ come from the sections of $\cC_{\cH_{g,n}/k}\to \cH_{g,n/k}$.
\begin{corollary}\label{the class equals tautological class}
	With notation as above. If $g\geq 3$, then the class $\kappa_x$ is equal to $\kappa_j$ or $-\kappa_j$ for some $j\in\{1,\ldots,n\}$. \qed
\end{corollary}

Recall that in the case of the universal curve $\cC_{g,n/k}\to \M_{g,n/k}$, the sections are exactly the tautological ones \cite{EaKr}. Hain's algebraic proof of this fact uses the weighted completion of the arithmetic mapping class groups \cite{hain2}. Here, we will modify Hain's idea that two sections occupying the same class in $H^1_\et(\M_{g,n/k}, \H_\Ql)$ are in fact equal. The problem of applying the result \cite[Lemma 13.1]{hain2} to our case is that the key $R$-representation $\Lambda^3_0H:=V_{[1^3]}(-1)$ appearing in $\Gr^W_{-1}\u_{g,n}$ does not appear in $\Gr^W_{-1}\mathfrak{v}_{g,n}$. 
The following result is the modification of \cite[Lemma 13.1]{hain2}, which allows us to conclude the analogous result for our case without the use of the representation $\Lambda^3_0H$.  
\begin{proposition}\label{no map}
	If $g\geq 3$, $n\geq 1$, and $r\geq 1$,  there is no $R$-invariant  Lie algebra homomorphism
	$$\Gr^W_{\bullet}\mathfrak{v}_{g,n}[r]/W_{-3}\to\Gr^W_{\bullet}\mathfrak{v}_{g,2}/W_{-3}$$
	that induces  the map $(u_1,\ldots, u_n)\mapsto (u_1, u_1)$ on $\Gr^W_{-1}$. 
\end{proposition}
\begin{proof} We identify $H$ with $H_1(C, \Z)\otimes\Ql.$ Recall that $C$ is the fiber of the universal curve over the geometric point $\etabar$. 
	For $u\in H$, denote by $u^{(j)}$ the corresponding element in the $j$th copy of $H$ in 
	$\Gr^W_{-1}\mathfrak{v}_{g,n}[r]= H_1\oplus\cdots\oplus H_n
	.$
	Fix a symplectic basis $a_1, b_1, \ldots, a_g, b_g$ for $H_1(C,\Z)$ and let $\langle~,~\rangle:\Lambda^2H_1(C,\Z)\to\Z$ be the intersection paring. 	Suppose that such a homomorphism $\phi$ exists. From the exact sequence of $\D_{g,m}[r]$-modules
	$$0\to \p_{g,n}\to \mathfrak{v}_{g,m}[r]\to \mathfrak{v}_g[r]\to 0,$$ for a positive integer $m$, there is a decomposition
	$$\Gr^W_{-2}\mathfrak{v}_{g,m}[r]=\bigoplus_{1\leq i<j\leq m}\Ql(1)_{ij}\oplus\bigoplus_{j=1}^m\Gr^W_{-2}\p_j \oplus
	\Gr^W_{-2}\mathfrak{v}_g[r],$$
	where $\Ql(1)_{ij}$ is spanned by $\sum_{k=1}^g [a_k^{(i)},b_k^{(j)}]$. Denote the element $\sum_{k=1}^g [a_k^{(i)},b_k^{(j)}]$ in $\Gr^W_{-2}\mathfrak{v}_{g,m}[r]$ by $\Theta_{ij}$.
	We claim first that $\phi$ vanishes on the $\Ql(1)$ component. For any $i<j$ and $u,v\in H$, the bracket $[u^{(i)},v^{(j)}]$ is computed in \cite[\S12]{hain0} and is given by
	$$[u^{(i)},v^{(j)}]=\frac{\langle u, v\rangle}{g}\sum_{k=1}^g [a_k^{(i)},b_k^{(j)}]\,\,\,\,\text{in}\,\,\,\Gr^W_{-2}\mathfrak{v}_{g,m}[r].$$
	For $1<j$, we have
	$\phi(v^{(j)})=(0,0)$ in $\text{Gr}^W_{-1}\mathfrak{v}_{g,2}$, and hence $[\phi(u^{(i)}),\phi(v^{(j)})]=0$. Since $\phi$ is a homomorphism, it follows that
	$\phi(\Theta_{ij})=0$, and therefore $\phi$ vanishes on $\Ql(1)_{ij}$ for all $1\leq i<j\leq n $.\\
	\indent Next, we will compute $\sum_{k=1}^g[\phi(a_k^{(1)}),\phi(b_k^{(1)})]$ in $\Gr^W_{-2}\mathfrak{v}_{g,2}$. Denote the element $\sum_{k=1}^g[a_k^{(i)},b_k^{(i)}]$ in $\Gr^W_{-2}\mathfrak{v}_{g,m}[r]$ by $\Theta_i$. Theorem 12.6 in \cite{hain0} implies that we have
	\begin{align*}
	\phi(\Theta_1)=\sum_{k=1}^g[\phi(a_k^{(1)}),\phi(b_k^{(1)})]&=\sum_{k=1}^g[a_k^{(1)}+a_k^{(2)},b_k^{(1)}+ b_k^{(2)}]\\
	&=\sum_{k=1}^g\left([a_k^{(1)},b_k^{(1)}]+[a_k^{(2)},b_k^{(2)}]+ \frac{2}{g}\Theta_{12}\right)\\
	&=\left(\Theta_1+\Theta_2+ 2\Theta_{12}\right) \,\,\,\text{in}\,\,\,\Gr^W_{-2}\mathfrak{v}_{g,2}.
	\end{align*}
	But one has the relation \cite[Thm. 12.6]{hain0}
	\[\Theta_i+\frac{1}{g}\sum_{j\not=i}\Theta_{ij}=0,\hspace{.5in}\text{for $1\leq i\leq m$},\]
	in $\Gr^W_{-2}\mathfrak{v}_{g,m}[r]$, so in $\Gr^W_{-2}\mathfrak{v}_{g,2}$
	\[\phi(\Theta_1)=\Theta_1+\Theta_2+2\Theta_{12}=\frac{-1}{g}\Theta_{12}+\frac{-1}{g}\Theta_{12}+ 2\Theta_{12}=\frac{2g-2}{g}\Theta_{12}\not = 0.\]
	Therefore we have reach a contradiction.
\end{proof}

\begin{proof}[Proof of Theorem 1]
Let $x$ be a section of $\cC_{\cH_{g,n}/k}\to \cH_{g,n/k}$. By Corollary \ref{the class equals tautological class}, we may assume that $\kappa_x=\kappa_j$ or $-\kappa_j$ for some $j\in \{1,\ldots, n\}$. Without loss of generality, we may assume that $\kappa_x=\kappa_1$. Note that if $\kappa_x=-\kappa_1$, then the class $\kappa_{J\circ x}$ of the hyperelliptic conjugate of the section $x$ denoted by $J\circ x$  is equal to $\kappa_1$. Let $r\geq 3$. By pulling the sections $x$ and $x_1$, we consider them as sections of $\cC_{\cH_{g,n/k}}[r]\to \cH_{g,n/k}[r],$ denoted also by $x_1$ and $x$. The corresponding classes in $H^1_\et(\cH_{g,n/k}[r], \H_\Ql)$ are equal. Let $\mathrm{Jac}\to \cH_{g,n/k}[r]$ be the relative Jacobian of the family $\cC_{\cH_{g,n/k}}[r]\to \cH_{g,n/k}[r]$. By \cite[Cor. 12.4]{hain2}, the class $[x_1]-[x]$ is a torsion in $\Jac(\cH_{g,n/k}[r])$. Say $[x_1]-[x]=t$. If the torsion $t=0$, then $x=x_1$. If $t\not=0$, then $x_1$ and $x$ are disjoint. This implies that there is an induced $k$-morphism $\Phi:\cH_{g,n/k}[r]\to \cH_{g,2/k}$ defined by $[C]\mapsto [C;x_1, x]$. The morphism $\Phi$ induces an $R$-invariant Lie algebra map $\phi:\Gr^W_\bullet\mathfrak{v}_{g,n}[r]/W_{-3}\to \Gr^W_{\bullet}\mathfrak{v}_{g,2}/W_{-3}$.
That $\kappa_x=\kappa_1$ implies that the map $\phi$ on $\Gr^W_{-1}$ is given by $(u_1,u_2,\ldots,u_n)\mapsto (u_1,u_1)$, but this is impossible by Proposition \ref{no map}. Thus we have $t=0$. The sections $x_1$ and $x$ are equal over $\cH_{g,n/k}[r]$ and hence over $\cH_{g,n/k}$. This completes the proof of Theorem 1.

\end{proof}
\begin{proof}[Proof of Theorem 2]
	By the universal property of weighted completion, each splitting of the extension 
	$$1\to \pi_1(C_\etabar, \bar x)\to \pi_1(\cC_{\cH_{g/k}}, \bar x)\to \pi_1(\cH_{g/k}, \etabar)\to 1$$ induces a $\D^{\cC}_{g}$-module section of the exact sequence
	$$0\to \p\to \mathfrak{v}^\cC_{g}\to \mathfrak{v}_{g}\to 0$$ and hence an $R$-invariant Lie algebra section of $\beta_0:\h_{g,1}\to \h_{g}$. By Proposition \ref{induced sections}, there is no such section for $\beta_0$. Therefore, we are done. 	
	\end{proof}
	\begin{corollary}
		Let $k$ be a field of characteristic zero and let $\ell$ be a prime number. If the $\ell$-adic cyclotomic character $\chi_\ell:G_k\to \Z^\times_\ell$ is infinite and if $g\geq 3$, then the section conjecture holds for the universal hyperelliptic curve 
		$\cC_{\cH_{g/k}}\to \cH_{g/k}$.  \qed
	\end{corollary}
\section{Non-abelian cohomology of $\pi_1(\cH_{g,n/k})$}
In this section, first we will briefly review the non-abelian cohomology of proalgebraic groups developed in \cite{hain4} and then compute the non-abelian cohomology of $\Delta^\arith_{g,n}:=\pi_1(\cH_{g,n/k},\etabar)$ with the coefficient group given by the continuous unipotent completion $\cP$ of $\pi_1(C, \bar x)$ over $\Ql$, where $C$ is the fiber of the universal family over $\etabar$. 
\subsection{Non-abelian cohomology schemes $H^1_\nab(\cG, \cN)$}
\indent Let $F$ be a field of characteristic zero. Let $S$ be a connected reductive group over $F$. We equip $S$ with a nontrivial central cocharacter $\omega:\Gm\to S$. We consider the following extension of proalgebraic groups:
$$1\to \cN\to \E\to\cG\to1.$$
Here $\cG$ is a negatively weighted extension of $S$ by a prounipotent $F$-group $\U$ and  $\E$ is an extension of $\cG$ by a unipotent $F$-group $\cN$. We observe that $\E$ is an extension of $S$ by a prounipotent $F$-group $\cV$, which we assume to be negatively weighted, i.e., the action of $\Gm$ on $H_1(\cV)$ has only negative weights. There is a commutative diagram of  prounipotent $F$-groups
$$\xymatrix{
	1\ar[r]&\cN\ar[r]\ar@{=}[d]&\cV\ar[r]\ar[d]&\U\ar[r]\ar[d]&1\\
	        1\ar[r]&\cN\ar[r]& \E\ar[r]&\cG\ar[r]&1,
	    }
$$
where the rows are exact.

Since $\E$ acts on the Lie algebra $\n$ of $\cN$ by adjoint action, the nilpotent Lie algebra $\n$ admits a natural weight filtration $W_\bullet \n$. The assumption on $\cV$ implies that we have $\n=W_{-1}\n$. We choose a lift $\hat{\omega}$ of $\omega$ to $\E$. The composition of $\hat \omega$ with $\E\to \cG$ is a lift $\tilde \omega$ of $\omega$ to $\cG$. Denote the Lie algebras of $\E$ and $\cG$ by $\e$ and $\g$, respectively. We consider the set of the sections of $\E\to\cG$. A section $s$ of $\E\to\cG$ is said to be $\hat \omega$-graded if $s\circ \tilde \omega=\hat \omega$. Let $A$ be an $F$-algebra.  Each section $s$ of $\E\otimes_FA\to \cG\otimes_FA$ induces a section $\Gr^W_\bullet ds$ of the associated graded Lie algebra surjection $\Gr^W_\bullet\e\otimes_FA\to \Gr^W_\bullet \g\otimes_FA$.  By \cite[Prop. 4.3]{hain4}, the functor taking $A$ to the set $\{\text{$S$-invariant graded sections of }\Gr^W_\bullet\e\otimes_FA\to \Gr^W_\bullet \g\otimes_FA\}$ is representable by an ind-affine scheme $\mathrm{Sect}^S_\bullet(\Gr^W_\bullet\g, \Gr^W_\bullet \n)$. The following result of Hain follows from the key fact that each $\cN$-conjugacy class of a section of $\E \to\cG$ contains a unique $\hat \omega$-graded section. 
\begin{theorem}[{\cite[Thm. 4.6, Cor. 4.7]{hain4}}]\label{nonab iso}There is an ind-affine scheme $H^1_\nab(\cG, \cN)$ that represents  the functor taking an $F$-algebra $A$ to the set of 
$\cN(A)$-conjugacy classes of sections of $\E \otimes_FA\to \cG\otimes_FA$.
If $H^1(\g, \Gr^W_m\n)$ is finite dimensional for all $m\in \Z$, then it is of finite type. Furthermore, there is an isomorphism of $F$-schemes $$H^1_\nab(\cG, \cN)\cong \mathrm{Sect}^S_\bullet(\Gr^W_\bullet\g, \Gr^W_\bullet \n).$$
\end{theorem}

When $\cN$ is prounipotent, assuming that each of the weight graded quotients $\Gr^W_\bullet H_1(\cN)$ is finite dimensional and $\mathrm{dim}~H^1(\g, \Gr^W_{m}\n)<\infty$ for all $m\in \Z$, we consider the  affine $F$-scheme $H^1_\nab(\cG, \cN)=\varprojlim H^1_\nab(\cG,\cN/W_m)$. It represents the functor in our interest.  In our case, we have $\cN=\cP$ and $\cG=\D_{g,n}$. In order to compute $H^1_\nab(\D_{g,n}, \cP)$, we use the following ``exact sequence" of non-abelian cohomology of proalgebraic groups. 
\begin{proposition}[{\cite[Prop. 4.8]{hain4}}] \label{exact seq} Let $N>1$. Assume that the finiteness condition $\mathrm{dim}~H^1(\g, \Gr^W_{-l}\n)<\infty$ for all $1\leq l \leq N$ holds. Then 
	\begin{enumerate}
	\item 	there exists a morphism of affine $F$-schemes 
	$$H^1_\nab(\cG, \cN/W_{-l})\overset{\delta}\to H^2(\g, \Gr^W_{-l}\n),$$
	\item   there exists a principal action of $H^1(\g, \Gr^W_{-l}\n)$ on $H^1_\nab(\cG, \cN/W_{-l-1})$, 
	\item 	for all $F$-algebras $A$, the set $H^1_\nab(\cG, \cN/W_{-l-1})(A)$ is a principal\\ $H^1(\g, \Gr^W_{-l}\n)(A)$-bundle over  $(\delta^{-1}(0))(A)$.
	\end{enumerate}
	
\end{proposition}
These  properties can be put together into an ``exact sequence"
$$H^1(\g, \Gr^W_{-l}\n)\curvearrowright H^1_\nab(\cG, \cN/W_{-l-1})\overset{p}\to H^1_\nab(\cG, \cN/W_{-l})\overset{\delta}\to H^2(\g, \Gr^W_{-l}\n),$$
where the map $p$ is induced by the quotient $\cN/W_{-l-1}\to\cN/W_{-l}$. More precisely, a section $s$ in $H^1_\nab(\cG, \cN/W_{-l})(A)$ induces a graded section $ ds_\ast$ of 
$$0\to\Gr^W_\bullet\n/W_{-l}\n\otimes A\to \Gr^W_\bullet\e/W_{-l}\n\otimes A\to\Gr^W_\bullet\g\otimes A\to0.$$
The section $s$ lifts to a section in $H^1_\nab(\cG, \cN/W_{-l-1})(A)$ if and only if the graded section $ ds_\ast$ lifts to a section of
$$0\to\Gr^W_\bullet\n/W_{-l-1}\n\otimes A\to \Gr^W_\bullet\e/W_{-l-1}\n\otimes A\to\Gr^W_\bullet\g\otimes A\to0.$$
Composing $ ds_\ast$ with the unique $S$-module section of $$\Gr^W_\bullet\e/W_{-l-1}\n\otimes A\to\Gr^W_\bullet\e/W_{-l}\n\otimes A,$$ we obtain an $S$-module section $\tilde{ds}_\ast$ of $\Gr^W_\bullet\e/W_{-l-1}\n\otimes A\to\Gr^W_\bullet\g\otimes A$. The obstruction for this section to lift to  a Lie algebra section is given by $\delta(s)$. When $\delta(s)=0$,   $ds_\ast$ lifts to a section $dt_\ast$ and for each $z\in H^1(\g, \Gr^W_{-l}\n)$, $dt_\ast+z$ is a graded Lie algebra section of $\Gr^W_\bullet\e/W_{-l-1}\n\otimes A\to\Gr^W_\bullet\g\otimes A$ that lifts $ds_\ast$.

\subsection{Proof of Theorem 3} Our proof of Theorem 3 is essentially a modification of the proof for \cite[Thm.~3]{hain2} to the universal hyperelliptic curves. 
 Let $p\not=\ell$ be prime numbers. Suppose that $k$ is a number field or a finite extension of $\Qp$. For such a field $k$, the finiteness condition for $H^1(\d_{g,n}, \Gr^W_m\p)$ will be satisfied for all $m\in\Z$. The following result is the modification of \cite[Prop. 17.3]{hain2} to our case.
\begin{proposition}\label{cup projuct is injective}
		If $g\geq 3$ and $n\geq 1$, then the map
	$$H^1(\d_{g,n}, H)\otimes H^1(\d_{g,n}, \Gr^W_{-l}\p)\to H^2(\d_{g,n}, \Gr^W_{-l-1}\p)$$
	induced by the bracket $H\otimes\Gr^W_{-l}\p\to \Gr^W_{-l-1}\p$ is injective for each $l\geq 2$. 
\end{proposition}
\begin{proof} Recall that $R=\GSp(H)$. Let $S=\Sp(H)\subset R$. 
	The connectivity of $R$ implies that there is an isomorphism\\ $H^j(\d_{g,n}, V)\cong \Hom_{R}(H_j(\mathfrak{v}_{g,n}), V)$ for all finite dimensional $R$-representation $V$. Let $\Delta^\arith_{g,n}=\pi_1(\cH_{g,n/k}, \etabar)$. Consider the diagram
	$$\xymatrix{
		H^1(\d_{g,n}, H)\otimes H^1(\d_{g,n},\Gr^W_{-l}\p)\ar[d]\ar[r]& H^2(\d_{g,n}, \Gr^W_{-l-1}\p)\ar[d]\\
		H^1(\Delta^\arith_{g,n}, H)\otimes H^1(\Delta^\arith_{g,n}, \Gr^W_{-l}\cP)\ar[r]& H^2(\Delta^\arith_{g,n}, \Gr^W_{-l-1}\cP).
	}$$
	The left vertical map is an isomorphism and the one on the right is an injection by \cite[Prop. 6.8.]{hain2}. Since the above diagram commutes, it will suffice to show that the bottom horizontal map is an injection. Now, for each $l\geq 2$, we have the  $R$-decomposition $\Gr^W_{-l}\cP=(\Gr^W_{-l}\cP)^{S}\oplus V$ of $\Gr^W_{-l}\cP$ into a sum of the trivial $S$-isotypical component and an $R$-representation $V$ of weight $-l$ whose isotypical components are geometrically nontrivial. \\
	\indent Firstly, we claim that $H^1(\Delta^\arith_{g,n}, V)=0$. We have an isomorphism 
	$$H^1(\Delta^\arith_{g,n}, V)\cong\Hom_{R}(\Gr^W_{-l}H_1(\mathfrak{v}_{g,n}), V).$$ The exact sequence 
	$$0\to \p_{g,n}\to \mathfrak{v}_{g,n}\to\mathfrak{v}_g\to 0$$ induces the exact sequence 
	$$\Gr^W_{\bullet}H_1(\p_{g,n})\to \Gr^W_\bullet H_1(\mathfrak{v}_{g,n})\to \Gr^W_{\bullet} H_1(\mathfrak{v}_g)\to 0.$$ This implies that $\Gr^W_{-2}H_1(\mathfrak{v}_{g,n})= \Gr^W_{-2}H_1(\mathfrak{v}_g)$. Now, it follows from Proposition \ref{even weight} and Lemma \ref{no nichi}  that $\Hom_{R}(\Gr^W_{-l}H_1(\mathfrak{v}_{g,n}), V)=0$ for each $l\geq 2$. Thus our claim holds. \\
\indent	Secondly we claim that there is an isomorphism 
	$$H^1(\Delta^\arith_{g,n}, (\Gr^W_{-l}\cP)^{S})\cong H^1(G_k, (\Gr^W_{-l}\cP)^{S})$$ for each $l\geq 2$. Consider the exact sequence $1\to \pi_{g,n}\to \Delta^\geom_{g,n}\to \Delta^\geom_g\to 1.$  This extension gives a spectral sequence $E^{s,t}_2=H^s(\Delta^\geom_g, H^t(\pi_{g,n}, \Ql))\Rightarrow H^{s+t}(\Delta^\geom_{g,n}, \Ql)$. Hence there is an exact sequence
	$$(*)\hspace{.5in}0\to H^1(\Delta^\geom_g, H^0(\pi_{g,n}, \Ql))\to H^1(\Delta^\geom_{g,n}, \Ql)\to\hspace{2in}$$
	$$ \hspace{1in}H^0(\Delta^\geom_g, H^1(\pi_{g,n}, \Ql))\to H^2(\Delta^\geom_g, H^0(\pi_{g,n},\Ql)).$$ The first term vanishes and so does the third term, since $H^1(\pi_{g,n}, \Ql)\cong \oplus_{j=1}^nH_j$ and the monodromy image of $\Delta^\geom_g$ in $\Sp(H)$ is Zariski-dense. Therefore, we have $H^1(\Delta^\geom_{g,n}, \Ql)=0$. Our claim then follows from an easy spectral sequence argument applied to the exact sequence $1\to \Delta^\geom_{g,n}\to \Delta^\arith_{g,n}\to G_k\to 1$. \\
\indent 	Thirdly, we observe that $H^1(\Delta^\arith_{g,n}, H)=H^1(\Delta^\geom_{g,n},H)$. To see this, consider the above exact sequence $(*)$ with the coefficient group $H$ replacing $\Ql$. Since $\pi_{g,n}$ acts on $H$ trivially and the hyperelliptic involution acts on $H$ as $-\id$, the first and fourth terms vanish. Now we note that there are isomorphisms 
$$H^1(\pi_{g,n}, H)^{\Delta^\geom_g}\cong Hom_{S}(\oplus_{j=1}^nH_j, H)=\oplus_{j=1}^n\Hom_{S}(H_j, H)=\oplus_{j=1}^n\Ql.$$ This shows that  $H^1(\Delta_{g,n}^\geom, H)\cong \oplus_{j=1}^n\Ql$. We have seen that $H^1(\Delta^\arith_{g,n}, H)=\oplus_{j=1}^n\Ql\kappa_j$. Since there is an isomorphism $H^1(\Delta^\arith_{g,n}, H)=H^0(G_k, H^1(\Delta^\geom_{g,n}, H))$, our third claim holds. In particular, this shows that $H^1(\Delta^\geom_{g,n}, H)$ is a trivial $G_k$-module. \\
\indent Finally, we prove the main claim. The rest of the proof goes in the same manner as in \cite[Cor.16.4]{hain4}. For completeness, we include the argument. The center-freeness of the graded Lie algebra $\Gr^W_\bullet \cP$ implies that the commutator map $H\otimes (\Gr^W_{-l}\cP)^{S}\to \Gr^W_{-l-1}\cP$ is an $R$-invariant injection. Hence there is an injection $H^2(\Delta^\arith_{g,n}, H\otimes(\Gr^W_{-l}\cP)^{S})\to H^2(\Delta^\arith_{g,n},\Gr^W_{-l-1}\cP)$.
	Since the map 	$H^1(\Delta^\arith_{g,n}, H)\otimes H^1(\Delta^\arith_{g,n}, \Gr^W_{-l}\cP)\to H^2(\Delta^\arith_{g,n}, \Gr^W_{-l-1}\cP)$ factors through this injection, it will suffice to show that the cup product map	$$H^1(\Delta^\arith_{g,n}, H)\otimes H^1(\Delta^\arith_{g,n}, (\Gr^W_{-l}\cP)^{S})\to H^2(\Delta^\arith_{g,n}, H\otimes(\Gr^W_{-l}\cP)^{S})$$
	 is an injection. We have
	\begin{align*}
		&H^1(\Delta^\arith_{g,n}, H)\otimes H^1(\Delta^\arith_{g,n}, (\Gr^W_{-l}\cP)^{S})\\
	\cong&H^1(\Delta_{g,n}^\geom, H)\otimes H^1(G_k, (\Gr^W_{-l}\cP)^{S})\\
	\cong&H^1(G_k, H^1(\Delta^\geom_{g,n}, H)\otimes (\Gr^W_{-l}\cP)^{S})\\
	\cong&H^1(G_k, H^1(\Delta^\geom_{g,n}, H\otimes (\Gr^W_{-l}\cP)^{S}))\\
	\hookrightarrow& H^2(\Delta^\arith_{g,n},H\otimes (\Gr^W_{-l}\cP)^{S})
	\end{align*}
	The last injection follows from the fact that the bottom  row of the spectral sequence 
	$E^{s,t}=H^s(G_k,H^t(\Delta^\geom_{g,n}, V))\Rightarrow H^{s+t}(\Delta^\arith_{g,n}, V)$ vanishes for a geometrically nontrivial $R$-representation $V$. 
\end{proof}	

	Recall that a section of the universal hyperelliptic curve $\cC_{\cH_{g,n}/k}\to \cH_{g,n/k}$ induces a section  $s_x$ of $\D_{\cC_{g,n}}\to \D_{g,n}$. For each $l\geq 1$, $W_{-l}\cP$ is the $l$-th term of the lower central series of $\cP$, and so it is a normal subgroup of $\D_{\cC_{g,n}}$. Thus there is an exact sequence $1\to \cP/W_{-l-1}\cP\to \D_{\cC_{g,n}}/W_{-l-1}\cP\to \D_{g,n}\to 1$. The section $s_x$ induces a section $(s_x)_l$ of 
	$$\D_{\cC_{g,n}}/W_{-l-1}\cP\to \D_{g,n}$$ for each $l\geq1$. Recall that the sections $x_i$ are the tautological ones and the $J\circ x_i$ are their hyperelliptic conjugates. For each $i=1,\ldots,n$, denote $s_{x_i}$ by $s_i$ and $s_{J\circ x_i}$ by $s_i^\sigma$.  The following result is a reinterpretation of Theorem \ref{sections of lie algebras} in terms of the non-abelian cohomology scheme of proalgebraic groups.
	\begin{proposition}\label{nonab sections}
	 If $g\geq 3$ and $n\geq 1$, then the map 
	 $$i_4:\cC_{\cH_{g,n}/k}(\cH_{g,n/k})\to H^1_\nab(\D_{g,n}, \cP/W_{-5})(\Ql),\,\,\,\, x\mapsto (s_x)_4$$ is a bijection. 
	\end{proposition}

\begin{proof}
By Theorem 1, the sections of $\cC_{\cH_{g,n}/k} \to \cH_{g,n/k}$ are the tautological ones $x_1, \ldots, x_n$ and their hyperelliptic conjugates $J\circ x_1, \ldots, J\circ x_n$. By Theorem \ref{nonab iso}, there is a bijection $H^1_\nab(\D_{g,n}, \cP/W_{-5})(\Ql)\cong \Sect(\Gr^W_\bullet\mathfrak{v}_{g,n}, \Gr^W_\bullet(\p/W_{-5}))(\Ql)$. It then follows from Theorem \ref{sections of lie algebras} that $\Sect(\Gr^W_\bullet\mathfrak{v}_{g,n}, \Gr^W_\bullet(\p/W_{-5}))(\Ql)$ are given by the sections $\zeta_1^\pm, \ldots, \zeta_n^\pm.$
The derivative map $\Gr^W_\bullet d$ takes $(s_i)_4$ to $\zeta_i^{+}$ and $(s_{i}^\sigma)_4$ to $\zeta^{-}_i$ for each $i$. 
	
\end{proof}	
We need the analogue of the results \cite[Prop. 15.1 \& 15.2]{hain2}. The same proofs work for our case as well.
\begin{proposition}
	For each $\Ql$-algebra $A$, there is an exact sequence 
	$$1\to \cP(A)\to \Delta_{\cC}^A\to \Delta^\arith_{g,n}\to 1$$
	that is a pullback of the exact sequence
	$$1\to \cP(A)\to \D_{\cC_{g,n}}(A)\to \D_{g,n}(A)\to 1$$
	along the representation $\Delta^\arith_{g,n}\to \D_{g,n}(A)$. \qed
	\end{proposition}	
	
	For each $l\geq1$, pushing along the surjection $\cP(A)\to (\cP/W_{-l-1}\cP)(A)$, we obtain an extension 
	$$1\to (\cP/W_{-l}\cP)(A)\to \Delta^A_{\cC}/(W_{-l}\cP)(A)\to \Delta^\arith_{g,n}\to 1.$$
	Define $H^1_\nab(\Delta^\arith_{g,n}, \cP/W_{-l})(A)$ to be the set of $(\cP/W_{-l})(A)$-conjugacy classes of the sections of $\Delta^A_{\cC}/(W_{-l}\cP)(A)\to \Delta^\arith_{g,n}$.
	\begin{proposition}
		The representation $\Delta^\arith_{g,n}\to \D_{g,n}(\Ql)$ induces a bijection
		$$H^1_\nab(\D_{g,n}, \cP/W_{-l})(\Ql)\to H^1_\nab(\Delta^\arith_{g,n}, \cP/W_{-l})(\Ql).$$ 
		\qed
	\end{proposition}
	Now we have the following key fact that is our analogue of \cite[Prop.~18.5]{hain2} whose  proof can be easily modified for our case with Propositions \ref{cup projuct is injective} and \ref{nonab sections} along induction using Proposition \ref{exact seq}.
	\begin{proposition}Let $l\geq 4$.
		If $g\geq 3$, then there are natural isomorphisms
		
		$$H^1_\nab(\Delta^\arith_{g,n}, \cP/W_{-l-1})\cong H^1_\nab(\D_{g,n}, \cP/W_{-l-1})
		                                    $$
		                                    $$\cong
		                                      \left\{
		                                    \begin{array}{ll}
		                                    \{s_1, s^\sigma_1, \ldots, s_n,s^\sigma_n\}\times \Spec\Ql & l \text{ odd} \\
		                                    \{s_1, s^\sigma_1, \ldots, s_n,s^\sigma_n\}\times H^1(G_k,(\Gr^W_{-l}\cP)^{\Sp(H)}) & l \text{ even}
		                                    \end{array}
		                                    \right.
		                                   	   $$
    where $(s_i)_l=s_i$ for $l$ odd and $(s_i)_l= (s_i, 0)$ for $l$ even. \qed
		                                   	    
	\end{proposition}

Finally, since $H^1_\nab(\Delta^\arith_{g,n}, \cP)\cong \varprojlim H^1_\nab(\Delta^\arith_{g,n},  \cP/W_{-l})$, this completes the proof of Theorem 3.


\begin{thebibliography}{}
	\bibitem{Acamp}
	N.~A'Campo:
	{\em Tresses, monodromie et le groupe symplectique}, Comment. Math. Helv., 54(2): 318-327, 1979.
	\bibitem{M.A}
	M.~Anderson:
	{\em Exactness properties of pro finite completion functors}, Topology 13 (1974) 229-239.
	
	\bibitem{AV}
	A.~Arsie and A.~Vistoli:
	{\em Stacks of cyclic covers of projective spaces}, Math. 140 (2004), 647-666.
 
	\bibitem{ak}
	M.~Asada and M.~Kaneko:
	{\em On the automorphism group of some pro-$\ell$ fundamental groups in Galois Representations and Arithmetic Algebraic Geometry}, editor Y.~Ihara, Advanced Studies in Pure Mathematics 12 (1987), 137-159. MR 89h: 20053.
	
	\bibitem{BiHi}
	J.~Birman and H.~Hilden:
	{\em On the mapping class groups of closed surfaces as covering spaces}, in Advances in the Theory of Riemann surfaces, Ann. of Math. Stud. 66 (1971) 81-115.

	\bibitem{bor1}
	A.~Borel:
	{\em Stable real cohomology of arithmetic groups}, Ann. Sci. Ecole Norm. Sup. 7 (1974), 235-272.
	
	\bibitem{bor2}
	{\em Stable real cohomology of arithmetic groups II}, Manifolds and Groups, Papers in Honor of Yozo Matsushima, Progress in Mathematics 14, Birkhauser, Boston, 1981, 21-55.
	
	\bibitem{BMP}
	T.~Brendle, D.~Margalit, and A.~ Putman: 
	{\em Generators for the hyperelliptic Torelli group and the kernel of the Burau representation at $t=-1$}, Inventiones Mathematicae, 1-48, Springer, 2014.
	
	\bibitem{DGMS}
	P.~Deligne, P.~Griffiths, J.~ Morgan, D.~Sullivan: {\em Real homotopy theory of K$\ddot{\text{a}}$hler manifolds}, Invent. Math. 29 (1975), 245-274.
	
	\bibitem{DM}
	P.~Deligne and D.~Mumford:
	{\em The irreducibility of the space of curves of given  genus}, Publ. Math. Inst. Hautes \'Etudes Sci. 36 (1969) 75-109.
	
	\bibitem{ear}
	C.~Earle:
	{\em On the moduli of closed Riemann surfaces with symmetries}, Ann. of Math. Studies, No. 66, Princeton Univ. 1971,  119-130.
	
	\bibitem{EaKr}
	C.~Earle, I.~Kra: {\em On sections of some holomorphic families of closed Riemann surfaces}, Acta Math. 137 (1976), 49-79.
	
	\bibitem{ful}
	W.~Fulton and J.~Harris:
	{\em Representation theory. A first course.}, Graduate Texts in Mathematics,~129,~Springer-Verlag,~1991.

	
	\bibitem{hain0}
	R.~Hain:
	{\em Infinitestimal presentations of Torelli groups}, J.~Amer.~Math.~Soc. 10 (1997), 597-651.
	
	
	\bibitem{hain2}
	R.~Hain:
	{\em Rational points of universal curves}, J. Amer. Math. Soc. 24 (2011), 709-769.
	
	\bibitem{hain3}
	R.~Hain:
	{\em Relative weight filtrations on completions of mapping class groups}, in Groups of Diffeomorphisms, Advanced Studies in Pure Mathematics, vol. 52 (2008), pp.~309-368, Mathematical Society of Japan.
	
	\bibitem{hain4}
	{\em Remarks on non-abelian cohomology of proalebraic groups},Journal of Algebraic Geometry, vol. 22 (2013), pp. 581-598, ISSN 1056-3911 [arXiv:1009.3662].
	
	\bibitem{hm1}
	R.~Hain and M.~Matsumoto:
	{\em Galois actions on fundamental groups of curves and the cycle $C-C^{-}$}, J.~Inst.~Math.~Jussieu 4 (2005), 363-403.
	
	\bibitem{wei}
	R.~Hain and M.~Matsumoto:
	{\em Weighted completion of Galois groups and Galois actions on the fundamental group of $\P^1-\{0,1,\infty\}$}, Compositio Math. 139 (2003), 119-167.
	
	\bibitem{Hu}
	J.~Hubbard: {\em Sur la non-existence de sections analytiques $\grave{a}$ la courbe univerelle de Teichm$\ddot{u}$ller}, C.~R.~Acad. Sci. Paris S$\acute{\text{e}}$r. A-B 274 (1972), A978-A979. 
	


	\bibitem{joh1}
	D.~Johnson:
	{\em An abelian quotient of the mapping cass group $\T_g$}, Math. Ann., 249(3): 225-242 1980.
	\bibitem{joh}
	D.~Johnson:
	{\em The structure of the Torelli group, III: The abelianization of $\T$}, Topology 24 (1985), 127-144.


	\bibitem{Kim}
	M.~Kim:
	{\em The motivic fundamental group of $\P^1-\{0,1,\infty\}$ and the theorem of Siegel}, Invent. Math. 161 (2005), 629-656.

	
	
	\bibitem{mss}
	S.~Morita, T.~Saito, and M.~Suzuki:
	{\em Abelianizations of derivation Lie algebras of the free associative algebra and the free Lie algebra}, Duke Math J.,Vol.~62, Number 5 (2013) 965-1002.
	
	
	
	
	\bibitem{noo2}
	B.~Noohi:
	{\em Fundations of topological stacks I}, submitted, 2005, [arXiv:math/0503247].
	
	
	\bibitem{Qui}
	D.~Quillen:
	{\em Rational homotopy theory}, Ann.~of Math. 90(1969), 205-295.
		
	\bibitem{tanaka}
	A.~Tanaka:
	{\em The first homology group of the hyperelliptic mapping class group
		with twisted coefficients}, Topology and its Applications 115 (2001) 19–42.

\end{thebibliography}
\end{document}